\documentclass[11pt]{article}

\usepackage{amsmath}
\usepackage{amssymb}
\usepackage{amsthm}
\usepackage{amsxtra}
\usepackage{enumitem}
\usepackage[text={0.75\paperwidth,0.75\paperheight}]{geometry} 
\usepackage[colorlinks=true]{hyperref}
\usepackage{mathtools}
\usepackage{titlesec}
\usepackage[all]{xy}

\usepackage[parfill]{parskip}
\begingroup
\makeatletter
   \@for\theoremstyle:=definition,remark,plain\do{%
     \expandafter\g@addto@macro\csname th@\theoremstyle\endcsname{%
        \addtolength\thm@preskip\parskip
     }%
   }
\endgroup

\DeclareMathAlphabet{\mathpzc}{OT1}{pzc}{m}{it} 

\DeclareFontFamily{U}{MnSymbolC}{}
\DeclareSymbolFont{mnsymbols}{U}{MnSymbolC}{m}{n}
\DeclareFontShape{U}{MnSymbolC}{m}{n}{
    <-6>  MnSymbolC5
   <6-7>  MnSymbolC6
   <7-8>  MnSymbolC7
   <8-9>  MnSymbolC8
   <9-10> MnSymbolC9
  <10-12> MnSymbolC10
  <12->   MnSymbolC12}{}
 \DeclareMathSymbol{\boxtimes}{2}{mnsymbols}{117}
\DeclareMathOperator{\boxtie}{\mathbin{\Join}}

\xyoption{tips}
\SelectTips{cm}{10}
\UseTips

\DeclareSymbolFont{symbolsA}{U}{txsya}{m}{n}
\DeclareSymbolFont{symbolsC}{U}{txsyc}{m}{n}
\DeclareMathSymbol{\multimap}{\mathrel}{symbolsA}{40}
\DeclareMathSymbol{\multimapinv}{\mathrel}{symbolsC}{18}

\swapnumbers \theoremstyle{plain}
\newtheorem{lem}[subsubsection]{Lemma}
\newtheorem{prop}[subsubsection]{Proposition}
\newtheorem{thm}[subsubsection]{Theorem}
\newtheorem*{thm*}{Theorem}
\newtheorem{cor}[subsubsection]{Corollary}
\theoremstyle{definition}

\newtheorem{ex}[subsubsection]{Example}
\newtheorem{exs}[subsubsection]{Examples}
\newtheorem{rem}[subsubsection]{Remark}
\newtheorem{rems}[subsubsection]{Remarks}

\titleformat{\subsection}[runin]{\normalfont\normalsize\bfseries}{\thesubsection}{0.4em}{}[.]
\titleformat{\subsubsection}[runin]{\normalfont\normalsize\bfseries}{\thesubsubsection}{0.4em}{}[.]
\titleformat{\paragraph}[runin] {\normalfont\normalsize\bfseries}{\theparagraph}{0.4em}{}[.]
\titleformat{\subparagraph}[runin] {\normalfont\normalsize\bfseries}{\thesubparagraph}{0.4em}{}[.]

\titlespacing*{\subsection} {0pt}{3.25ex plus 1ex minus .2ex}{0.4em} 
\titlespacing*{\subsubsection}{0pt}{3.25ex plus 1ex minus .2ex}{0.4em} 
\titlespacing*{\paragraph} {0pt}{3.25ex plus 1ex minus .2ex}{0.4em} 
\titlespacing*{\subparagraph} {\parindent}{3.25ex plus 1ex minus .2ex}{0.4em}

\newenvironment{examples}[1][\mbox{}]
{\begin{exs}#1\vspace{-.3em}\begin{enumerate}[label=(\arabic{*})]}{\end{enumerate}\end{exs}}

\newenvironment{remarks}[1][\mbox{}]
{\begin{rems}#1\vspace{-.3em}\begin{enumerate}}{\end{enumerate}\end{rems}}

\renewcommand{\to}{\xrightarrow{\rule{1.45ex}{0ex}}} 

\newcommand{\dand}{\qquad\text{and}\qquad} 
\newcommand{\df}{\emph} 

\renewcommand{\epsilon}{\varepsilon}

\newcommand{\inv}{{-1}} 
\newcommand{\ob}{\mathop{\mathrm{ob}}\nolimits} 
\newcommand{\op}{\mathrm{op}} 

\newcommand{\V}{V} 

\newcommand{\cat}[1]{\mathsf{#1}}

\newcommand{\cC}{\cat{C}} 
\newcommand{\cX}{\cat{X}} 
\newcommand{\cY}{\cat{Y}} 
\newcommand{\cZ}{\cat{Z}} 
\newcommand{\AbGrp}{\cat{AbGrp}} 
\newcommand{\Mod}{\cat{Mod}} 
\newcommand{\Mon}{\cat{Mon}} 
\newcommand{\Qnt}{\cat{Qnt}} 
\newcommand{\Rng}{\cat{Rng}} 
\newcommand{\Set}{\cat{Set}} 
\newcommand{\Sup}{\cat{Sup}} 

\newcommand{\mon}[1]{\mathbb{#1}}

\newcommand{\mAb}{{\mon{A}}} 
\newcommand{\mId}{\mon{I}} 
\newcommand{\mP}{\mon{P}} 
\newcommand{\mS}{\mon{S}} 
\newcommand{\mT}{\mon{T}} 


\begin{document}

\title{Tensors, monads and actions}


\author{Gavin J. Seal}

\date{June 22, 2013}



\maketitle

\begin{center}
\large Dedicated to the memory of Pawe{\l} Waszkiewicz
\end{center}

\bigskip

\begin{abstract}
We exhibit sufficient conditions for a monoidal monad $\mT$ on a monoidal category $\cC$ to induce a monoidal structure on the Eilenberg--Moore category $\cC^\mT$ that represents bimorphisms. The category of actions in $\cC^\mT$ is then shown to be monadic over the base category $\cC$.
\end{abstract}


\nocite{AdaHerStr:90}

\setcounter{section}{-1}
\section{Introduction}

The original motivation for the current work stemmed from the observation that the category $\cC^M$ of actions of a monoid $M$ in a monoidal category $(\cC,\otimes,E)$ is \emph{monadic} over its base (see Proposition~\ref{prop:ActionEM} below). When the base monoidal category is itself an Eilenberg--Moore category $\cC^\mT$, the composition of the forgetful functors $(\cC^\mT)^M\to\cC^\mT$ and $\cC^\mT\to\cC$ is monadic again in classical examples: the category of $R$-modules, seen as a category of $R$-actions in the category $\Set^\mAb\cong\AbGrp$ of abelian groups, is monadic over $\Set$, and the category of actions of an integral quantale in the category $\Set^\mP\cong\Sup$ of sup-semilattices is monadic over $\Set$ \cite{PedTho:89}. These instances suggest the following underlying principle:
\begin{quote}
The category $(\cC^\mT)^M$ of actions of a monoid $M$ in a monoidal Eilenberg--Moore category $\cC^\mT$ is monadic over $\cC$.
\end{quote}
In order to \emph{define} actions in $\cC^\mT$, we first need a tensor $(-\boxtimes-)$ on $\cC^\mT$ that  encodes the ``bilinear'' nature of the action morphism $M\boxtimes X\to X$. The providential structure is provided by a \emph{monoidal monad} on $(\cC,\otimes,E)$ that allows the introduction of morphisms in $\cC$ that are ``$\mT$-algebra homomorphisms in each variable'', as originally suggested in \cite{Lin:69}.

Let us say a word on the technical setting we adopted for this work. In \cite{Koc:70,Koc:71b,Koc:71a,Koc:72}, Kock presents the fundamentals of symmetric monoidal monads in a context of closed categories. However, closedness does not appear to play an explicit role in the classical construction of the tensor product on $\AbGrp$ or $\Sup$. It also seems reasonable to aim for an action morphism that occurs as an algebraic structure $M\otimes A\to A$ on $A$, rather than as a morphism of monoids $M\to[A,A]$ (where $[-,-]$ would designate an internal hom). Hence, we chose to follow \cite{Gui:80} and consider a base category whose monoidal structure is neither assumed to be symmetric, nor closed. Alas, the result we needed in \textit{op.cit.} is presented with a somewhat obscure proof (in particular, the proposed construction of the unit isomorphism in $\cC^\mT$ seems a bit brusque). This lead us to our current version of Theorem~\ref{thm:BoxtimesMonoidal} that, in turn, provided the necessary ingredients to prove the sought result in Theorem~\ref{thm:ActionsAreMonadic}. We note that the hypotheses of Theorem~\ref{thm:BoxtimesMonoidal} involve conditions on the tensor $(-\boxtimes-)$ that might not be practical to verify, as the latter is realized via a coequalizer. To remedy this, Theorem~\ref{thm:BoxtimesMonoidal:practical} proposes hypotheses that can be tested directly on the base category in the non-closed case (to be compared with Corollary~\ref{cor:thm:BoxtimesMonoidal}).

Our work is thus structured as follows. In Section \ref{sec:Basics}, we establish the relevant definitions and notations pertaining to monoidal monads. We also recall that these are fundamentally linked to monoidal structures of Kleisli categories. In Section \ref{sec:Monoidal}, we recall how bimorphisms and the tensor on $\cC^\mT$ induced by a monoidal monad $\mT$ are related. Proposition~\ref{prop:CoeqEqGuitart} then exhibits the link between the tensor proposed in \cite{Lin:69} with the one studied in \cite{Gui:80}. We also review some useful facts about reflexive coequalizers, and recall in \ref{conv:Tensor} that the tensor in $\cC^\mT$ of free algebras is the free algebra of their tensor in $\cC$ (see \cite[Proposition~21]{Gui:80}). This observation is crucial to establish that the tensor on $\cC^\mT$ is associative, a fact that is proved in Theorem~\ref{thm:BoxtimesMonoidal}, and again in Theorem~\ref{thm:BoxtimesMonoidal:practical} with alternative hypotheses. We then consider algebraic functors $\cC^\mT\to\cC^\mS$ induced by monoidal monad morphisms $\mS\to\mT$, and show that they are themselves monoidal. Once $\cC^\mT$ is equipped with the adequate monoidal structure, we turn our attention to actions in Section~\ref{sec:Actions}. Our main result is, as mentioned above, that the monadic functors $(\cC^\mT)^M\to\cC^\mT$ and $\cC^\mT\to\cC$ compose to form yet another monadic functor $(\cC^\mT)^M\to\cC$. We conclude by showing that the classical restriction-of-scalars functor between categories of modules is algebraic.

Throughout the text, we illustrate the various notions introduced with the classical examples mentioned above, that is, with structures related to the free abelian group and the powerset monads. Example~\ref{ex:BoxtimesMonoidal:practical} also demonstrates that binary coproducts in $\cC^\mT$ can be interpreted as the tensor induced by binary coproducts in $\cC$. Of course, these examples are far from being exhaustive, but we feel that they adequately represent the concepts developed, while hinting at further applications. 


\section{Basic structures}\label{sec:Basics}

\subsection{Monoidal categories}
Let $\cC$ be a monoidal category, with its tensor denoted by $(-\otimes-):\cC\times\cC\to\cC$, its unit by $E$, and its structure natural isomorphisms by
\[
\alpha_{X,Y,Z}:(X\otimes Y)\otimes Z\to X\otimes(Y\otimes Z)\ ,\quad\lambda_X:E\otimes X\to X\ ,\quad\rho_X:X\otimes E\to X\ .
\]
When the monoidal category $\cC$ is symmetric, we denote by $\sigma_{X,Y}:X\otimes Y\to X\otimes Y$ the components of its braiding natural isomorphisms (see \cite{Mac:71} or \cite{Bor:94b}). We will denote a monoidal category $(\cC,\otimes,\alpha,E,\lambda,\rho)$ or $(\cC,\otimes,\alpha,E,\lambda,\rho,\sigma)$ more briefly by $(\cC,\otimes,E)$.

\begin{ex}
The monoidal category recurrent in most of our examples is the category $(\Set,\times,\{\star\})$ of sets and maps with its cartesian structure.
\end{ex}

\subsection{Monoidal monads}
Let $(\cC,\otimes,E)$ be a monoidal category. A functor $T:\cC\to\cC$, with a map $\eta_E:E\to TE$, and a family of maps $\kappa=(\kappa_{X,Y}:TX\otimes TY\to T(X\otimes Y))_{X,Y\in\ob\cC}$ natural in $X$ and $Y$, is \df{monoidal} if $(T,\eta_E,\kappa)$ is compatible with the associativity and unitary natural transformations of $(\cC,\otimes,E)$, so that one has for all $X,Y,Z\in\ob\cC$:

\begin{enumerate}
\item\label{cond:0} $\kappa_{X,Y\otimes Z}\cdot(1_{TX}\otimes\kappa_{Y,Z})\cdot\alpha_{TX,TY,TZ}=T\alpha_{X,Y,Z}\cdot\kappa_{X\otimes Y,Z}\cdot(\kappa_{X,Y}\otimes 1_{TZ})$, that is, the diagram
\[
\xymatrix{(TX\otimes TY)\otimes TZ\ar[d]_{\alpha}\ar[r]^-{\kappa\otimes1}&T(X\otimes Y)\otimes TZ\ar[r]^-{\kappa}&T((X\otimes Y)\otimes Z)\ar[d]^{T\alpha}\\
TX\otimes(TY\otimes TZ)\ar[r]^-{1\otimes\kappa}&TX\otimes T(Y\otimes Z)\ar[r]^-{\kappa}&T(X\otimes(Y\otimes Z))}
\]
commutes;
\item\label{cond:1} $T\lambda_{X}\cdot\kappa_{E,X}\cdot(\eta_E\otimes1_{TX})=\lambda_{TX}$ and $T\rho_{X}\cdot\kappa_{X,E}\cdot(1_{TX}\otimes\eta_E)=\rho_{TX}$, that is, the diagrams
\[
\xymatrix{E\otimes TX\ar[dr]_{\lambda}\ar[r]^-{\eta_E\otimes 1}&TE\otimes TX\ar[r]^-{\kappa}&T(E\otimes X)\ar[dl]^-{T\lambda}\\
&TX&}
\xymatrix{\mbox{}\ar@{}[d]|{\textstyle{\text{and}}}\\
\mbox{}}
\xymatrix{TX\otimes E\ar[dr]_{\rho}\ar[r]^-{1\otimes\eta_E}&TX\otimes TE\ar[r]^-{\kappa}&T(X\otimes E)\ar[dl]^-{T\rho}\\
&TX&}
\]
commute.
\end{enumerate}

A monad $\mT=(T,\mu,\eta)$ is \df{monoidal} if the functor $T:\cC\to\cC$ comes with a family of maps $\kappa=(\kappa_{X,Y}:TX\otimes TY\to T(X\otimes Y))_{X,Y\in\ob\cC}$, natural in $X$ and $Y$, that make $(T,\kappa,\eta_E)$ a monoidal functor, and such that the following conditions are satisfied for all $X,Y\in\ob\cC$:

\begin{enumerate}
\setcounter{enumi}{2}
\item\label{cond:3} $\mu_{X\otimes Y}\cdot T\kappa_{X,Y}\cdot\kappa_{TX,TY}=\kappa_{X,Y}\cdot(\mu_X\otimes\mu_Y)$, that is, the diagram
\[ 
\xymatrix@C=4em{TTX\otimes TTY\ar[r]^-{T\kappa\cdot\kappa}\ar[d]_-{\mu\otimes\mu}&TT(X\otimes Y)\ar[d]^-{\mu}\\
TX\otimes TY\ar[r]^{\kappa}&T(X\otimes Y)}
\]
commutes;
\item\label{cond:compatibility} $\kappa_{X,Y}\cdot(\eta_X\otimes\eta_Y)=\eta_{X\otimes Y}$, that is, the diagram
\[
\xymatrix{&X\otimes Y\ar[dl]_{\eta\otimes\eta}\ar[dr]^{\eta}&\\
TX\otimes TY\ar[rr]^-{\kappa}&&T(X\otimes Y)}
\]
commutes.
\end{enumerate}

In the case where $(\cC,\otimes,E)$ is symmetric, the monoidal monad $(\mT,\kappa)$ is itself \df{symmetric} if the following condition is satisfied for all $X,Y\in\ob\cC$:
\begin{enumerate}
\setcounter{enumi}{4}
\item\label{cond:sym} $\kappa_{Y,X}\cdot\sigma_{TX,TY}=T\sigma_{X,Y}\cdot\kappa_{X,Y}$, that is, the diagram
\[
\xymatrix{TX\otimes TY\ar[d]_{\sigma}\ar[r]^{\kappa}&T(X\otimes Y)\ar[d]^{T\sigma}\\
TY\otimes TX\ar[r]^{\kappa}&T(Y\otimes X)}
\]
commutes.
\end{enumerate}

\begin{examples}\label{ex:MonoidalMonads}
\item The identity monad $\mId=(1_\cC,1,1)$ on a monoidal category $\cC$ is a monoidal monad via the identity natural transformation $(1_{X\otimes Y}:X\otimes Y\to X\otimes Y)_{X,Y\in\ob\cC}$. The monoidal monad is symmetric whenever $\cC$ is.
\item\label{ex:MonoidalMonads:Ab} The free abelian group monad $\mAb=(A,\sum,(-))$ on $(\Set,\times,\{\star\})$ is a symmetric monoidal monad via the natural transformation $\kappa$ whose component $\kappa_{X,Y}:AX\times AY\to A(X\times Y)$ sends a pair $\big(\sum_{x\in X}n_x\cdot x,\sum_{y\in Y}n_y\cdot y\big)$ (with all coefficients $n_i$ integers) to the element $\sum_{x\in X,y\in Y}(n_x+n_y)\cdot(x,y)$.
\item\label{ex:MonoidalMonads:P} The powerset monad $\mP=(P,\bigcup,\{-\})$ on $\Set$ (with its cartesian structure) is a symmetric monoidal monad via the natural transformation $\iota$ whose component $\iota_{X,Y}:PX\times PY\to P(X\times Y)$ sends a pair of subsets $(A,B)$ to their product $A\times B\subseteq X\times Y$.
\item\label{ex:MonoidalMonads:coprod} Any monad $\mT$ on a category $\cC$ whose monoidal structure is given by finite coproducts (so $\otimes=+$ and $E=\varnothing$ is the initial object in $\cC$) is monoidal with respect to the connecting $\cC$-morphisms $\kappa_{A,B}:TA+TB\to T(A+B)$.
\end{examples}

Monoidal monads correspond to monoidal structures of the Kleisli category, as follows.

\begin{prop}\label{prop:MonoidalLift}
Given a monad $\mT$ on a monoidal category $\cC$, there is a bijective correspondence between the following data:
\begin{enumerate}
\item families $\kappa=(\kappa_{X,Y}:TX\otimes TY\to T(X\otimes Y))_{X,Y\in\ob\cC}$ of $\cC$-morphisms natural in $X$ and $Y$ making $\mT$ a monoidal monad;
\item monoidal structures on the Kleisli category $\cC_\mT$ such that the left adjoint functor $F_\mT:\cC\to\cC_\mT$ is strict monoidal.
\end{enumerate}
Moreover, $(\mT,\kappa)$ is symmetric precisely when the corresponding monoidal structure on $\cC_\mT$ is symmetric.
\end{prop}

\begin{proof}
Given a $\cC$-morphism $\kappa_{Z,W}$, and $\cC$-morphisms $f:X\to TZ$, $g:Y\to TW$, one can define the $\cC$-morphism $f\otimes_\mT g:=\kappa_{Z,W}\cdot(f\otimes g):X\otimes Y\to T(Z\otimes W)$, thus equipping $\cC_\mT$ with a tensor $\otimes_\mT$ for which $F_\mT$ is strict monoidal. Conversely, if $\cC_\mT$ is monoidal with a tensor $\otimes_\mT$, then strict monoidality of $F_\mT$ forces equality $X\otimes_\mT Y=X\otimes Y$ for all $\cC$-objects $X$ and $Y$, and one can define $\kappa_{X,Y}:=1_{TX}\otimes_\mT 1_{TY}$. (See for example~\cite[Proposition 8]{Gui:80}.)
\end{proof}

\subsection{Monoidal monad morphisms}
Let $(\mS,\iota)$ and $(\mT,\kappa)$ be monoidal monads on a monoidal category $(\cC,\otimes,E)$. A monad morphism $\phi:\mS\to\mT$ is \df{monoidal} if, for all $X,Y\in\ob\cC$, the equality $\phi_{X\otimes Y}\cdot\iota_{X,Y}=\kappa_{X,Y}\cdot(\phi_X\otimes\phi_Y)$ holds, or equivalently, the diagram
\[
\xymatrix{SX\otimes SY\ar[r]^-{\iota}\ar[d]_{\phi\otimes\phi}&S(X\otimes Y)\ar[d]^{\phi}\\
TX\otimes TY\ar[r]^-{\kappa}&T(X\otimes Y)}
\]
commutes.

\begin{ex}\label{ex:MonoidalMonadsMorph}
For any monoidal monad $(\mT,\kappa)$ on $(\cC,\otimes,E)$, the unit $\eta:(\mId,1)\to(\mT,\kappa)$ is a monoidal monad morphism.
\end{ex}

The correspondence of Proposition~\ref{prop:MonoidalLift} extends to morphisms.

\begin{prop}\label{prop:MonoidalMorphismLift}
Given a monad $\mT$ on a monoidal category $\cC$, there is a bijective correspondence between the following data:
\begin{enumerate}
\item monoidal monad morphisms $\phi:\mS\to\mT$;
\item strict monoidal functors $L:\cC_\mS\to\cC_\mT$ that commute with the left adjoint functors from $\cC$:
\[
\xymatrix{\cC_\mS\ar[rr]^{L}&&\cC_\mT\\
&\cC\ar[ul]^{F_\mS}\ar[ur]_{F_\mT}\rlap{\ .}&}
\]
\end{enumerate}
\end{prop}

\begin{proof}
The one-to-one correspondence between monad morphisms and functors between the Kleisli categories is standard (see for example \cite[Theorem~2.2.2]{ManMul:07}): a monad morphism $\phi:\mS\to\mT$ defines a functor $L:\cC_\mS\to\cC_\mT$ that is the identity on objects and sends a $\cC_\mS$-morphism $f:X\to SY$ to the $\cC_\mT$-morphism $\phi_Y\cdot f:X\to TY$; conversely, a functor $L$ as in (2) defines a monad morphism $\phi:\mS\to\mT$ via its components $\phi_X:=L(1_{SX}):SX\to TY$. One easily verifies that if $\phi$ is monoidal, then $L$ is strict monoidal, and that the converse holds, too. 
\end{proof}


\section{The monoidal structure of $\cC^\mT$}\label{sec:Monoidal}

The prototypical tensor product that we wish to study is provided by the tensor product of $R$-modules. The role of this tensor is to represent bilinear maps. In our setting, the monoidal structure of the monad allows the introduction of the notion of such a ``morphism in each variable'' (as suggested in \cite{Lin:69} and \cite{Koc:71b}).

\subsection{Bimorphisms}\label{ssec:Bim}
Let $(\mT,\kappa)$ be a monoidal monad on a monoidal category $(\cC,\otimes,E)$, and denote by $\cC^\mT$ its category of Eilenberg--Moore algebras. For $\mT$-algebras $(A,a)$, $(B,b)$, and $(C,c)$, we say that a $\cC$-morphism $f:A\otimes B\to C$ is a \df{$\cC^\mT$-bimorphism}\footnote{These bimorphisms should not be confused with the $\cC$-morphisms that are at the same time monic and epic. Since we do not consider the latter, there is little reason to avoid the bimorphism terminology.} (or simply a \df{bimorphism}) if the diagram
\begin{equation}\label{eq:bimorphism}
\xymatrix{TA\otimes TB\ar[d]_{a\otimes b}\ar[r]^-{\kappa}&T(A\otimes B)\ar[r]^-{Tf}&TC\ar[d]^{c}\\
A\otimes B\ar[rr]^-{f}&&C}
\end{equation}
commutes. The set of all bimorphisms $f:(A,a)\otimes (B,b)\to(C,c)$ is denoted by $\cC^\mT(A,B;C)$. 

\begin{examples}
\item For the identity monad $\mId=(1_\cC,1,1)$ on a monoidal category $\cC$, with the identity natural transformation, $\cC^\mId$-bimorphisms are just $\cC$-morphisms since $\cC^\mId\cong\cC$.
\item For the free abelian group monad $\mAb$, the category $\Set^\mAb$ is isomorphic to $\AbGrp$, the category of abelian groups. Via the natural transformation $\kappa$ of Example \ref{ex:MonoidalMonads}\ref{ex:MonoidalMonads:Ab}, a $\Set^\mAb$-bimorphism $f:A\times B\to C$ is a map that is additive in each variable.
\item For the powerset monad $\mP$ on $\Set$ with the natural transformation $\iota$ of Example~\ref{ex:MonoidalMonads}\ref{ex:MonoidalMonads:P}, the Eilenberg--Moore category $\Set^\mP$ is isomorphic to $\Sup$, the category of complete sup-semilattices. With this interpretation, a $\Set^\mP$-bimorphism $f:A\times B\to C$ is a map that preserves suprema in each variable.
\end{examples}

The next result shows in what sense a bimorphism captures the idea of a ``morphism in each variable''. 

\begin{prop}\label{prop:BimTwoParts}
Let $(\mT,\kappa)$ be a monoidal monad on a monoidal category $(\cC,\otimes,E)$. For $\mT$-algebras $(A,a)$, $(B,b)$, and $(C,c)$, a $\cC$-morphism $f:A\otimes B\to C$ is a $\cC^\mT$-bimorphism if and only if both diagrams
\[
\xymatrix@C=3em{A\otimes TB\ar[d]_{1\otimes b}\ar[r]^-{\kappa\cdot(\eta\otimes 1)}&T(A\otimes B)\ar[r]^-{Tf}&TC\ar[d]^{c}\\
A\otimes B\ar[rr]^-{f}&&C}
\qquad\xymatrix{\mbox{}\ar@{}[d]|{\textstyle{\text{and}}}\\\mbox{}}\qquad
\xymatrix@C=3em{TA\otimes B\ar[d]_{a\otimes 1}\ar[r]^-{\kappa\cdot(1\otimes\eta)}&T(A\otimes B)\ar[r]^-{Tf}&TC\ar[d]^{c}\\
A\otimes B\ar[rr]^-{f}&&C}
\]
commute.
\end{prop}

\begin{proof}
If $f$ is a $\cC^\mT$-bimorphism, then one can append each of the following commutative diagrams
\[
\xymatrix{A\otimes TB\ar[r]^-{\eta\otimes1}\ar[dr]_{1\otimes b}&TA\otimes TB\ar[d]^-{a\otimes b}\\
&A\otimes B}
\qquad\xymatrix{\mbox{}\ar@{}[d]|{\textstyle{\text{and}}}\\\mbox{}}\qquad
\xymatrix{TA\otimes B\ar[r]^-{1\otimes\eta}\ar[dr]_{a\otimes1}&TA\otimes TB\ar[d]^-{a\otimes b}\\
&A\otimes B}
\]
to the left of \eqref{eq:bimorphism} to obtain the respective diagrams of the statement. Conversely, if the two diagrams of the statement commute, then
\begin{align*}
f\cdot(a\otimes b)&=f\cdot(1_A\otimes b)\cdot(a\otimes 1_{TB})\\
&=c\cdot Tf\cdot\kappa_{A,B}\cdot(\eta_A\otimes 1_{TB})\cdot(a\otimes 1_{TB})\\
&=c\cdot Tf\cdot\kappa_{A,B}\cdot(Ta\otimes T1_B)\cdot(\eta_{TA}\otimes 1_{TB})\\
&=c\cdot T(f\cdot(a\otimes 1_B))\cdot\kappa_{TA,B}\cdot(\eta_{TA}\otimes 1_{TB})\\
&=c\cdot T(c\cdot Tf\cdot\kappa_{A,B}\cdot(1_{TA}\otimes\eta_B))\cdot\kappa_{TA,B}\cdot(\eta_{TA}\otimes 1_{TB})\\
&=c\cdot Tf\cdot\mu_{A\otimes B}\cdot T\kappa_{A,B}\cdot T(1_{TA}\otimes\eta_B)\cdot\kappa_{TA,B}\cdot(\eta_{TA}\otimes 1_{TB})\\
&=c\cdot Tf\cdot\mu_{A\otimes B}\cdot T\kappa_{A,B}\cdot\kappa_{TA,TB}\cdot(T1_{TA}\otimes T\eta_B)\cdot(\eta_{TA}\otimes 1_{TB})\\
&=c\cdot Tf\cdot\kappa_{A,B}\cdot(\mu_A\otimes\mu_B)\cdot(\eta_{TA}\otimes T\eta_B)=c\cdot Tf\cdot\kappa_{A,B}\ .
\end{align*}
This shows commutativity of  \eqref{eq:bimorphism}.
\end{proof}

\begin{prop}\label{prop:BimHomFunctorial}
Let $(\mT,\kappa)$ be a monoidal monad on a monoidal category $(\cC,\otimes,E)$. Then
\[
\cC^\mT(-,-;-):(\cC^\mT\times\cC^\mT)^\op\times\cC^\mT\to\Set
\]
is a functor defined by
\[
\cC^\mT(g,h;k)(f):=k\cdot f\cdot(g\otimes h)
\]
for all $\cC^\mT$-morphisms $g:A'\to A$, $h:B'\to B$, $k:C\to C'$, and bimorphisms $f:A\otimes B\to C$.
\end{prop}

\begin{proof}
The given functor is well-defined. Indeed, for $\cC$-morphisms $f$, $g$, $h$, $k$ as in the claim (and $\mT$-algebras $(A,a)$, $(B,b)$, $(C,c)$, $(A',a')$, $(B',b')$, $(C',c')$), one has
\begin{align*}
c'\cdot T(k\cdot f\cdot(g\otimes h))\cdot\kappa_{A',B'}&=k\cdot c\cdot Tf\cdot\kappa_{A,B}\cdot(Tg\otimes Th)\\
&= k\cdot f\cdot(a\otimes b)\cdot(Tg\otimes Th)=k\cdot f\cdot (g\otimes h)\cdot (a'\otimes b')\ ,
\end{align*}
so that $\cC^\mT(g,h;k)f$ is in $\cC^\mT(A',B';C')$. Functoriality is immediate by monoidality of $\cC$.
\end{proof}

\subsection{Tensor products}\label{ssec:TensorProd}
Let $(\mT,\kappa)$ be a monoidal monad on a monoidal category $(\cC,\otimes,E)$. The \df{tensor product} $q_{A,B}:T(A\otimes B)\to(A\boxtimes B,a\boxtie b)$ of $\mT$-algebras $(A,a)$ and $(B,b)$ is given by the following coequalizer diagram
\[
\xymatrix@C=3em{T(TA\otimes TB)\ar[r]<0.7ex>^-{\mu\cdot T\kappa}\ar[r]<-0.7ex>_-{T(a\otimes b)}&T(A\otimes B)\ar[r]^{q}&A\boxtimes B}
\]
in $\cC^\mT$. We often assume that $q_{A,B}$ is implicit and speak of the \emph{tensor product} $(A\boxtimes B,a\bowtie b)$, or simply $A\boxtimes B$. For $\cC^\mT$-morphisms $g:(A,a)\to(A',a')$, $h:(B,b)\to(B',b')$, one has
\begin{align*}
q_{A',B'}\cdot T(g\otimes h)\cdot\mu_{A\otimes B}\cdot T\kappa_{A,B}&=q_{A',B'}\cdot\mu_{A',B'}\cdot TT(g\otimes h)\cdot T\kappa_{A,B}\\
&=q_{A',B'}\cdot\mu_{A',B'}\cdot T\kappa_{A',B'}\cdot T(Tg\otimes Th)\\
&=q_{A',B'}\cdot T(a'\otimes b')\cdot T(Tg\otimes Th)=q_{A',B'}\cdot T(g\otimes h)\cdot T(a\otimes b)\ ,
\end{align*}
so there is a unique $\cC^\mT$-morphism $g\boxtimes h:A\boxtimes B\to A'\boxtimes B'$ with
$(g\boxtimes h)\cdot q_{A,B}=q_{A',B'}\cdot T(g\otimes h)$, and the diagram
\[
\xymatrix{T(A\otimes B)\ar[r]^-{q}\ar[d]_{T(g\otimes h)}&A\boxtimes B\ar[d]^{g\boxtimes h}\\
T(A'\otimes B')\ar[r]^-{q}&A'\boxtimes B'}
\]
commutes.

We say that $\cC^\mT$ \df{has tensors} when $q_{A,B}$ exists for all $\mT$-algebras $A$ and $B$. In particular, if $\cC^\mT$ has reflexive coequalizers, then it has tensors. For $\cC^\mT$ to moreover be a \emph{monoidal} category, we also need associativity and unitary $\cC^\mT$-morphisms that make the corresponding coherence diagrams commute. This is the subject of Theorem~\ref{thm:BoxtimesMonoidal}.

\subsubsection{Convention}
For the sake of convenience, from now on we say that a diagram of the form
\[
\xymatrix@C=3em{X\ar[r]<0.7ex>^-{g}\ar[r]<-0.7ex>_-{h}&Y\ar[r]^{f}&Z}
\]
commutes if $f\cdot g=f\cdot h$.

\begin{prop}\label{prop:CoeqEqGuitart}
Given a monoidal monad $(\mT,\kappa)$ on a monoidal category $(\cC,\otimes,E)$ and $\mT$-algebras $(A,a)$, $(B,b)$, the following statements are equivalent for all $\cC$-morphisms $f:T(A\otimes B)\to C$:
\begin{enumerate}[label=\normalfont{(\roman*)}]
\item the diagram $\xymatrix@C=5.5em{T(TA\otimes TB)\ar[r]<0.7ex>^-{\mu\cdot T\kappa}\ar[r]<-0.7ex>_-{T(a\otimes b)}&T(A\otimes B)\ar[r]^-{f}&C}$ commutes;
\item the diagram $\xymatrix@C=5.5em{T(TA\otimes TB)\ar[r]<0.7ex>^-{\mu\cdot T(\kappa\cdot(\eta\cdot a\otimes 1))}\ar[r]<-0.7ex>_-{\mu\cdot T(\kappa\cdot(1\otimes\eta\cdot b))}&T(A\otimes B)\ar[r]^-{f}&C}$ commutes.
\end{enumerate}
\end{prop}

\begin{proof}
Suppose that $f\cdot \mu_{A\otimes B}\cdot T\kappa_{A,B}=f\cdot T(a\otimes b)$. Then one immediately obtains
\[
f\cdot\mu_{A\otimes B}\cdot T(\kappa_{A,B}\cdot(\eta_{A}\cdot a\otimes 1_{TB}))=f\cdot T(a\otimes b)=f\cdot\mu_{A\otimes B}\cdot T(\kappa_{A,B}\cdot(1_{TA}\otimes\eta_B\cdot b))\ .
\]
Conversely, suppose that $f\cdot\mu_{A\otimes B}\cdot T(\kappa_{A,B}\cdot(\eta_{A}\cdot a\otimes 1_{TB}))=f\cdot\mu_{A\otimes B}\cdot T(\kappa_{A,B}\cdot(1_{TA}\otimes\eta_B\cdot b))$ holds. One first notes
\begin{align*}
f\cdot T(a\otimes b)&=f\cdot\mu_{A\otimes B}\cdot T\eta_{A\otimes B}\cdot T(a\otimes b)\\
&=f\cdot\mu_{A\otimes B}\cdot T(\kappa_{A,B}\cdot(\eta_A\otimes\eta_B)\cdot(a\otimes b))\\
&=f\cdot\mu_{A\otimes B}\cdot T(\kappa_{A,B}\cdot(\eta_A\cdot a\otimes 1_{TB})\cdot(1_{TA}\otimes\eta_B\cdot b))\\
&=f\cdot\mu_{A\otimes B}\cdot T(\kappa_{A,B}\cdot(1_{TA}\otimes\eta_B\cdot b)\cdot(1_{TA}\otimes\eta_B\cdot b))\\
&=f\cdot\mu_{A\otimes B}\cdot T(\kappa_{A,B}\cdot(1_{TA}\otimes\eta_B\cdot b))\ ,
\end{align*}
so
\begin{align*}
f\cdot T(a\otimes 1_{B})&=f\cdot\mu_{A\otimes B}\cdot T(\kappa_{A,B}\cdot(1_{TA}\otimes\eta_B))\qquad\text{and}\\
f\cdot T(a\otimes b)&=f\cdot\mu_{A\otimes B}\cdot T(\kappa_{A,B}\cdot(\eta_{A}\cdot a\otimes 1_{TB}))\ .
\end{align*}
These equalities, together with the fact that
\begin{align*}
\kappa_{A,B}&=\kappa_{A,B}\cdot(\mu_A\otimes\mu_B)\cdot(\eta_{TA}\otimes T\eta_B)\\
&=\mu_{A\otimes B}\cdot T\kappa_{A,B}\cdot\kappa_{TA,TB}\cdot (1_{TTA}\otimes T\eta_B)\cdot(\eta_{TA}\otimes 1_{TB})\\
&=\mu_{A\otimes B}\cdot T(\kappa_{A,B}\cdot(1_{TA}\otimes\eta_B))\cdot\kappa_{TA,B}\cdot(\eta_{TA}\otimes 1_{TB})
\end{align*}
yield
\begin{align*}
f\cdot \mu_{A\otimes B}\cdot T\kappa_{A,B}&=f\cdot\mu_{A\otimes B}\cdot T(\mu_{A\otimes B}\cdot T(\kappa_{A,B}\cdot(1_{TA}\otimes\eta_B))\cdot\kappa_{TA,B}\cdot(\eta_{TA}\otimes 1_{TB}))\\
&=f\cdot\mu_{A\otimes B}\cdot T(\kappa_{A,B}\cdot(1_{TA}\otimes\eta_B))\cdot\mu_{TA\otimes B}\cdot T(\kappa_{TA,B}\cdot(\eta_{TA}\otimes 1_{TB}))\\
&=f\cdot T(a\otimes 1_B)\cdot\mu_{TA\otimes B}\cdot T(\kappa_{TA,B}\cdot(\eta_{TA}\otimes 1_{TB}))\\
&=f\cdot \mu_{A\otimes B}\cdot T(\kappa_{A,B}\cdot(\eta_{A}\cdot a\otimes 1_{TB}))=f\cdot T(a\otimes b)\ ,
\end{align*}
as required.
\end{proof}

\begin{cor}\label{cor:CoeqEqGuitart}
Let $(\mT,\kappa)$ be a monoidal monad on a monoidal category $(\cC,\otimes,E)$ such that $\cC^\mT$ has tensors. Then for $\mT$-algebras $(A,a)$ and $(B,b)$, the diagram
\[
\xymatrix@C=5.5em{T(TA\otimes TB)\ar[r]<0.7ex>^-{\mu\cdot T(\kappa\cdot(\eta\cdot a\otimes 1))}\ar[r]<-0.7ex>_-{\mu\cdot T(\kappa\cdot(1\otimes\eta\cdot b))}&T(A\otimes B)\ar[r]^-{q}&A\boxtimes B}
\]
is a coequalizer in $\cC^\mT$ (with $q_{A,B}$ as defined in {\normalfont\ref{ssec:TensorProd}}).
\end{cor}

\begin{proof}
This is an immediate consequence of the previous Proposition.
\end{proof}

\begin{remarks}\label{rem:BimTens}
\item Corollary \ref{cor:CoeqEqGuitart} confirms that the coequalizer suggested to define a tensor in \cite[Remark, Section 1]{Lin:69} is the same as the one appearing in \cite[Proposition 16]{Gui:80}.
\item\label{rem:BimTens:Triv} As remarked by Linton \cite{Lin:11}, the study of bimorphisms trivializes for certain monads. For example, suppose that $\mT$ is a monoidal monad on $(\Set,\times,\{\star\})$ whose $\mT$-algebras $X$ have two nullary operations $0,1:\{\star\}\to X$ and a binary one $(-)\ast(-):X\times X\to X$ such that
\[
0\ast x = 0\qquad\text{and}\qquad1\ast x = x
\]
for all $x\in X$. If $f:A\times B\to C$ is a bimorphism, then 
\[
f(a,1)=1\qquad\text{and}\qquad f(0,b)=0
\]
for all $a\in A$, $b\in B$, so that $1=f(0,1)=0$. Hence, if $c\in C$, then $c=1\ast c=0\ast c=0$. That is, the only bimorphisms $f:A\times B\to C$ are those for which $C$ is a singleton. (See also Remark~\ref{rem:TrivBimTens}.) This observation applies in particular to bimorphisms in the category of semirings.
\end{remarks}


\subsection{Representing bimorphisms}
A major motivation to the introduction of tensor products in categories of $R$-modules is the representation of bimorphisms. Here, we show that the tensor product of \ref{ssec:TensorProd} plays that role with respect to the bimorphisms of \ref{ssec:Bim}.
 
\begin{prop}\label{prop:BoxtimesFunct}
Let $(\mT,\kappa)$ be a monoidal monad on a monoidal category $(\cC,\otimes,E)$ such that $\cC^\mT$ has tensors. Then
\[
\cC^\mT(-\boxtimes-,-):(\cC^\mT\times\cC^\mT)^\op\times\cC^\mT\to\Set
\]
is a functor. 
\end{prop}

\begin{proof}
Immediate by functoriality of $(-\otimes-)$ and $T$ (see \ref{ssec:TensorProd} and Proposition~\ref{prop:BimHomFunctorial}).
\end{proof}

The relation between bimorphisms and the tensor product is the subject of the following results.

\begin{lem}\label{lem:BimTensor}
Let $(\mT,\kappa)$ be a monoidal monad on a monoidal category $(\cC,\otimes,E)$. For $\mT$-algebras $(A,a)$, $(B,b)$, $(C,c)$, the following statements are equivalent for any $\cC$-morphism $f:A\otimes B\to C$:
\begin{enumerate}[label=\normalfont{(\roman*)}]
\item $f$ is a bimorphism;
\item The diagram $\xymatrix@C=3em{T(TA\otimes TB)\ar[r]<0.7ex>^-{\mu\cdot T\kappa}\ar[r]<-0.7ex>_-{T(a\otimes b)}&T(A\otimes B)\ar[r]^-{c\cdot Tf}&C}$ commutes.
\end{enumerate}
\end{lem}

\begin{proof}
If $f$ is a bimorphism, then
\[
c\cdot Tf\cdot\mu_{A\otimes B}\cdot T\kappa_{A,B}=c\cdot Tc\cdot TTf\cdot T\kappa_{A,B}=c\cdot Tf\cdot T(a\otimes b)\ .
\]
Conversely, if the diagram in (ii) commutes, then
\[
c\cdot Tf\cdot\kappa_{A,B}=c\cdot Tf\cdot\mu_{A\otimes B}\cdot\eta_{T(A\otimes B)}\cdot\kappa_{A,B}=c\cdot Tf\cdot T(a\otimes b)\cdot\eta_{TA\otimes TB}=f\cdot (a\otimes b)\ ,
\]
so $f$ is a bimorphism.
\end{proof}

\begin{lem}\label{lem:bim}
Let $(\mT,\kappa)$ be a monoidal monad on a monoidal category $(\cC,\otimes,E)$ such that $\cC^\mT$ has tensors. For $\mT$-algebras $(A,a)$, $(B,b)$, and $(C,c)$, the following statements hold.
\begin{enumerate}
\item If $f:A\otimes B\to C$ is a bimorphism, then there is a unique $\cC^\mT$-morphism $\overline{f}:A\boxtimes B\to C$ such that
\[
\xymatrix{T(A\otimes B)\ar[d]_{q}\ar[r]^-{Tf}&TC\ar[d]^{c}\\
A\boxtimes B\ar[r]^-{\overline{f}}&C}
\]
commutes.
\item If $g:A\boxtimes B\to C$ is a $\cC^\mT$-morphism, then $g\cdot q_{A,B}\cdot\eta_{A\otimes B}:A\otimes B\to C$ is a bimorphism that induces $g$, that is, $\overline{g\cdot q_{A,B}\cdot\eta_{A\otimes B}}=g$.
\end{enumerate}
\end{lem}

\begin{proof}
Since $c\cdot Tf$ is a $\cC^\mT$-morphism, the first claim follows directly from Lemma \ref{lem:BimTensor} and the universal property of the coequalizer $q_{A,B}$.

Given a $\cC^\mT$-morphism $g:A\boxtimes B\to C$, one uses that $q_{A,B}$ is a coequalizer and that $g\cdot q_{A,B}:(T(A\otimes B),\mu_{A\otimes B})\to(C,c)$ is a $\cC^\mT$-morphism to write
\begin{align*}
g\cdot q_{A,B}\cdot\eta_{A\otimes B}\cdot(a\otimes b)&=g\cdot q_{A,B}\cdot T(a\otimes b)\cdot\eta_{TA\otimes TB}\\
&=g\cdot q_{A,B}\cdot \mu_{A\otimes B}\cdot T\kappa_{A,B}\cdot\eta_{TA\otimes TB}\\
&=g\cdot q_{A,B}\cdot \mu_{A\otimes B}\cdot \eta_{T(A\otimes B)}\cdot\kappa_{A,B}\\
&=g\cdot q_{A,B}\cdot \mu_{A\otimes B}\cdot T\eta_{A\otimes B}\cdot\kappa_{A,B}\\
&=c\cdot T(g\cdot q_{A,B})\cdot T\eta_{A\otimes B}\cdot\kappa_{A,B}=c\cdot T(g\cdot q_{A,B}\cdot\eta_{A\otimes B})\cdot\kappa_{A,B}\ ,
\end{align*}
showing that $g\cdot q_{A,B}\cdot\eta_{A\otimes B}$ is a bimorphism. Since
\[
c\cdot T(g\cdot q_{A,B}\cdot\eta_{A\otimes B})=g\cdot q_{A,B}\cdot\mu_{A\otimes B}\cdot T\eta_{A\otimes B}=g\cdot q_{A,B}\ ,
\]
the $\cC^\mT$-morphism induced by $g\cdot q_{A,B}\cdot\eta_{A\otimes B}$ is indeed $g$.
\end{proof}

\begin{prop}\label{prop:bim}
Let $(\mT,\kappa)$ be a monoidal monad on a monoidal category $(\cC,\otimes,E)$ such that $\cC^\mT$ has tensors. For $\mT$-algebras $(A,a)$, $(B,b)$, and $(C,c)$, there is a bijection
\[
\cC^\mT(A,B;C)\cong\cC^\mT(A\boxtimes B,C)
\]
natural in each variable.
\end{prop}

\begin{proof}
The required bijection is described in Lemma \ref{lem:bim}. Indeed, if $g:A\boxtimes B\to C$ is a $\cC^\mT$-morphism, then $f=g\cdot q_{A,B}\cdot\eta_{A\otimes B}$ is a bimorphism, to which corresponds the unique $\cC^\mT$-morphism $g:A\boxtimes B\to C$. Conversely, if $f:A\otimes B\to C$ is a bimorphism, then there is a $\cC^\mT$-morphism $\overline{f}:A\boxtimes B\to C$ such that $c\cdot Tf=\overline{f}\cdot q_{A,B}$; according to Lemma \ref{lem:bim}, one obtains in return a bimorphism
\[
g=\overline{f}\cdot q_{A,B}\cdot\eta_{A\otimes B}=c\cdot Tf\cdot\eta_{A\otimes B}=c\cdot\eta_C\cdot f=f\ .
\]
For a $\cC^\mT$-bimorphism $f:A\otimes B\to C$, and $\cC^\mT$-morphisms $g:A'\to A$, $h:B'\to B$, $k:C\to C'$, one has
\[
k\cdot\overline{f}\cdot(g\boxtimes h)\cdot q_{A',B'}=k\cdot\overline{f}\cdot q_{A,B}\cdot T(g\otimes h)=k\cdot c\cdot Tf\cdot T(g\otimes h)=c'\cdot Tk\cdot Tf\cdot T(g\otimes h)\ ,
\]
that is,
\[
\overline{k\cdot f\cdot(g\otimes h)}=k\cdot\overline{f}\cdot(g\boxtimes h)
\]
(by unicity of the induced $\cC^\mT$-morphism). Hence, the bijection is natural.
\end{proof}

\begin{rem}\label{rem:TrivBimTens}
In the trivial cases where the only bimorphisms are $\cC$-morphisms $f:A\otimes B\to I$ into the terminal object of $\cC$ (as in Remark~\ref{rem:BimTens}\eqref{rem:BimTens:Triv}), Proposition~\ref{prop:bim} shows that $A\boxtimes B\cong I$: the identity $\cC^\mT$-morphism $1_{A\boxtimes B}$ corresponds to a bimorphism $f:A\otimes B\to A\boxtimes B$, so that $A\boxtimes B\cong I$.
\end{rem}


\subsection{Reflexive coequalizers}

We recall here some basic results pertaining to reflexive coequalizers, and thus applying to the coequalizer $q$ defined in \ref{ssec:TensorProd}.

\begin{prop}\label{prop:TensorReflexiveCoeq}
A functor $F:\cX\times\cY\to\cZ$ preserves reflexive coequalizers if and only if $F(X,-):\cY\to\cZ$ and $F(-,Y):\cX\to\cZ$ preserves reflexive coequalizers for all $X\in\ob\cX$, $Y\in\ob\cY$.
\end{prop}

\begin{proof}
The necessity of the statement is immediate since $F(X,g)=F(1_X,g)$ and $F(f,Y)=F(f,1_Y)$ for all $X\in\ob\cX$, $Y\in\ob\cY$. For the sufficiency, see the dual of \cite[Corollary~1.2.12]{Joh:02a}.
\end{proof}

Since we wish to study the tensor products $A\boxtimes B$ defined via the reflexive $\cC^\mT$-coequalizer $q_{A,B}:T(A\otimes B)\to A\boxtimes B$, we recall below two results pertaining to the existence of such colimits.

\begin{prop}\label{prop:CreatReflCoeq}
Let $\mT$ be a monad on $\cC$. If $\cC$ has reflexive coequalizers and $T$ preserves them, then the forgetful functor $\cC^\mT\to\cC$ creates them.
\end{prop}

\begin{proof}
See \cite[Proposition~3]{Lin:69}.
\end{proof}

\begin{prop}
If $\mT$ is a regular monad on $\cC$, then $\cC^\mT$ has all colimits that exist in $\cC$.
\end{prop}

\begin{proof}
See \cite[Proposition~20.33]{AdaHerStr:90}.
\end{proof}


\subsection{The monoidal Eilenberg--Moore category}
Consider a monoidal monad  $(\mT,\kappa)$ on a monoidal category $(\cC,\otimes,E)$. Theorem~\ref{thm:BoxtimesMonoidal} presents sufficient conditions for the $\cC^\mT$-morphisms
\[
q_{A,B}:T(A\otimes B)\to A\boxtimes B
\]
to induce a monoidal structure on $\cC^\mT$. The proof of this result relies in part on the explicit description of the coequalizer $q_{TA,TB}:T(TA\otimes TB)\to TA\boxtimes TB$ defined in~\ref{conv:Tensor}.

\begin{lem}\label{lem:CoeqSimple}
Let $\mT$ be a monad on a category $\cC$. For $\cC$-morphisms $r:Z\to TY$, $s:Z\to Y$, and $p:TY\to X$,
\[
(p\cdot\mu_Y\cdot Tr=p\cdot Ts)\implies(p\cdot r=p\cdot\eta_Y\cdot s)\ .
\]
\end{lem}

\begin{proof}
One simply has
\[
p\cdot r=p\cdot\mu_{Y}\cdot\eta_{TY}\cdot r=p\cdot\mu_{Y}\cdot Tr\cdot\eta_{Z}=p\cdot Ts\cdot\eta_{Z}=p\cdot\eta_{Y}\cdot s\ .
\]
\end{proof}

\begin{prop}\label{prop:TensorFree}
Let $(\mT,\kappa)$ be a monoidal monad on a monoidal category $(\cC,\otimes,E)$. The following statements hold:
\begin{enumerate}
\item the diagram
\[
\xymatrix@C=3em{T(TTA\otimes TTB)\ar[r]<0.7ex>^-{\mu\cdot T\kappa}\ar[r]<-0.7ex>_-{T(\mu\otimes\mu)}&T(TA\otimes TB)\ar[r]^-{\mu\cdot T\kappa}&T(A\otimes B)}
\]
forms a coequalizer in $\cC^\mT$;
\item if $f:A\to A'$ and $g:B\to B'$ are $\cC$-morphisms, then $T(f\otimes g):T(A\otimes B)\to T(A'\otimes B')$ is the unique $\cC^\mT$-morphism that makes the diagram
\[
\xymatrix{T(TA\otimes TB)\ar[r]^-{\mu\cdot T\kappa}\ar[d]_{T(Tf\otimes Tg)}&T(A\otimes B)\ar[d]^{T(f\otimes g)}\\
T(TA'\otimes TB')\ar[r]^-{\mu\cdot T\kappa}&T(A'\otimes B')}
\]
commute.
\end{enumerate}
\end{prop}

\begin{proof}
By naturality of $\mu$ and the fact that $\mu\cdot\mu T=\mu\cdot T\mu$, one has
\begin{align*}
\mu_{A\otimes B}\cdot T\kappa_{A,B}\cdot\mu_{TA\otimes TB}\cdot T\kappa_{TA,TB}&=\mu_{A\otimes B}\cdot T\mu_{A\otimes B}\cdot TT\kappa_{A,B}\cdot T\kappa_{TA,TB}\\
&=\mu_{A\otimes B}\cdot T\kappa_{A,B}\cdot T(\mu_A\otimes\mu_B)\ ,
\end{align*}
so the diagram in the statement commutes. Suppose that $f:(T(TA\otimes TB),\mu_{TA\otimes TB})\to(C,c)$ is a $\cC^\mT$-morphism that makes 
\[
\xymatrix@C=3em{T(TTA\otimes TTB)\ar[r]<0.7ex>^-{\mu\cdot T\kappa}\ar[r]<-0.7ex>_-{T(\mu\otimes\mu)}&T(TA\otimes TB)\ar[r]^-{f}&C}
\]
commute. Setting $\overline{f}:=f\cdot T(\eta_A\otimes\eta_B)$, one has
\begin{align*}
\overline{f}\cdot \mu_{A\otimes B}\cdot T\kappa_{A,B}&=f\cdot\mu_{A\otimes B}\cdot T\kappa_{A,B}\cdot T(T\eta_A\otimes T\eta_B)\\
&=c\cdot Tf\cdot T\kappa_{A,B}\cdot T(T\eta_A\otimes T\eta_B)\\
&=c\cdot Tf\cdot T\eta_{TA\otimes TB}=f
\end{align*}
by using the equality $f\cdot\kappa_{TA,TB}=f\cdot\eta_{TA\otimes TB}\cdot(\mu_{A}\otimes\mu_{B})$ following from Lemma~\ref{lem:CoeqSimple}. Since $\mu_{A\otimes B}\cdot T\kappa_{A,B}$ is an epimorphism, the comparison $\cC^\mT$-morphism $\overline{f}$ is uniquely determined. 

The diagram in the concluding statement commutes by naturality of $\mu$ and $\kappa$, and the unicity of the induced map follows from the fact that $\mu_{A\otimes B}\cdot T\kappa_{A,B}$ is epic.
\end{proof}

\subsubsection{Identifying coequalizers}\label{conv:Tensor} In view of Proposition~\ref{prop:TensorFree}, we can set
\[
q_{TA,TB}:=\mu_{A\otimes B}\cdot T\kappa_{A,B}\dand TA\boxtimes TB:=T(A\otimes B)
\]
for all $A,B\in\ob\cC$. Hence, if $f:TA\to TA'$ and $g:TB\to TB'$ are $\cC$-morphisms, then
\[
Tf\boxtimes Tg=T(f\otimes g):TA\boxtimes TB\to TA'\boxtimes TB'\ .
\]
If moreover the coequalizer $q_{A,B}:T(A\otimes B)\to A\boxtimes B$ exists in $\cC^\mT$, it follows from commutativity of
\[
\xymatrix@C=3em{T(TA\otimes TB)\ar[r]^-{T(a\otimes b)}\ar[d]_{\mu\cdot T\kappa}&T(A\otimes B)\ar[d]^{q}\\
T(A\otimes B)\ar[r]^-{q}&A\boxtimes B}
\]
that
\[
q_{A,B}=a\boxtimes b\ .
\]
These identifications will be used without necessarily further mention in the proof of Theorem~\ref{thm:BoxtimesMonoidal}.

\begin{rem}
The following result states that, under certain hypotheses, $(\cC^\mT,\boxtimes,TE)$ becomes a monoidal category with associativity and unitary structure morphisms \emph{induced} by those of $(\cC,\otimes,E)$. The meaning of this term is made clear directly in the proof in an attempt to avoid a rather cumbersome direct definition.
\end{rem}

\begin{thm}\label{thm:BoxtimesMonoidal}
Let $(\mT,\kappa)$ be a monoidal monad on a monoidal category $(\cC,\otimes,E)$ such that all the coequalizers $q_{A,B}$ exist in $\cC^\mT$. If 
\begin{gather*}
\xymatrix@C=5em{T(TA\otimes TB)\boxtimes TTC\ar[r]<0.7ex>^-{(\mu\cdot T\kappa)\boxtimes\mu}\ar[r]<-0.7ex>_-{T(a\otimes b)\boxtimes Tc}&T(A\otimes B)\boxtimes TC\ar[r]^-{q\boxtimes c}&(A\boxtimes B)\boxtimes C}\\
\xymatrix@C=5em{TTA\boxtimes T(TB\otimes TC)\ar[r]<0.7ex>^-{\mu\boxtimes (\mu\cdot T\kappa)}\ar[r]<-0.7ex>_-{Ta\boxtimes T(b\otimes c)}&TA\boxtimes T(B\otimes C)\ar[r]^-{a\boxtimes q}&A\boxtimes(B\boxtimes C)}
\end{gather*}
are coequalizer diagrams and $q_{A,B}\boxtimes q_{C,D}$ are epimorphisms (for all $\mT$-algebras $(A,a)$, $(B,b)$, $(C,c)$ and $(D,d)$), then $(\cC^\mT,\boxtimes,TE)$ is a monoidal category with structure morphisms induced by those of $(\cC,\otimes,E)$.

Moreover, if $(\cC,\otimes,E)$ and $(\mT,\kappa)$ are symmetric monoidal, then so is $(\cC^\mT,\boxtimes,TE)$.
\end{thm}

\begin{proof} In this proof, we consider $\mT$-algebras $(A,a)$, $(B,b)$, $(C,c)$, $(D,d)$, as well as $(A',a')$, $(B',b')$, $(C',c')$.

\textit{Associativity.}
By \ref{conv:Tensor} and the definitions of $\alpha$ and $\kappa$, both the inner and outer squares in
\[
\xymatrix@C=5.5em{T((TA\otimes TB)\otimes TC)\ar[d]_{T\alpha}\ar[r]<0.7ex>^-{(\mu\cdot T\kappa)\boxtimes\mu}\ar[r]<-0.7ex>_-{T(a\otimes b)\boxtimes Tc}&T((A\otimes B)\otimes C)\ar[d]^{T\alpha}\\
T(TA\otimes(TB\otimes TC))\ar[r]<-0.7ex>_-{(\mu\cdot T\kappa)\boxtimes\mu}\ar[r]<0.7ex>^-{Ta\boxtimes T(b\otimes c)}&T(A\otimes(B\otimes C))}
\]
commute in $\cC^\mT$. By hypothesis, $q_{A,B}\boxtimes c$ and $a\boxtimes q_{B,C}$ are coequalizers, so there is a $\cC^\mT$-isomorphism $\overline{\alpha}_{A,B,C}:(A\boxtimes B)\boxtimes C\to A\boxtimes(B\boxtimes C)$ induced by $T\alpha$:
\begin{equation*}
\xymatrix@C=4.5em{T((A\otimes B)\otimes C)\ar[r]^-{q\boxtimes c}\ar[d]_{T\alpha}&(A\boxtimes B)\boxtimes C\ar[d]^{\overline{\alpha}}\\
T(A\otimes(B\otimes C))\ar[r]^-{a\boxtimes q}&A\boxtimes(B\boxtimes C)\rlap{\ .}}
\end{equation*}
Note that for $\cC^\mT$-morphisms $f:(A,a)\to(A',a')$, $g:(B,b)\to(B',b,)$, $h:(C,c)\to(C',c')$, the diagram
\begin{equation*}
\xymatrix@C=4em{T((A\otimes B)\otimes C)\ar[r]^-{q\boxtimes c}\ar[d]_{T((f\otimes g)\otimes h)}&(A\boxtimes B)\boxtimes C\ar[d]^{(f\boxtimes g)\boxtimes h}\\
T((A'\otimes B')\otimes C')\ar[r]^-{q\boxtimes c'}&(A'\boxtimes B')\boxtimes C'\rlap{\ ;}}
\end{equation*}
commutes, and there is a similar diagram for $f\boxtimes(g\boxtimes h)$. Naturality of $T\alpha$, commutativity of the diagrams for $(f\boxtimes g)\boxtimes h$, $f\boxtimes(g\boxtimes h)$ and $\overline{\alpha}_{A,B,C}$ together with the fact that $q_{A,B}\boxtimes c$ is epic yield naturality of $\overline{\alpha}$:
\[
\xymatrix{(A\boxtimes B)\boxtimes C\ar[d]_{(f\boxtimes g)\boxtimes h}\ar[r]^-{\overline{\alpha}}&A\boxtimes(B\boxtimes C)\ar[d]^{f\boxtimes(g\boxtimes h)}\\
(A'\boxtimes B')\boxtimes C'\ar[r]^-{\overline{\alpha}}&A'\boxtimes(B'\boxtimes C')\rlap{\ .}}
\]
Commutativity of the coherence diagram
\[
\xymatrix{((A\boxtimes B)\boxtimes C)\boxtimes D\ar[r]^{\overline{\alpha}}\ar[d]_{\overline{\alpha}\boxtimes 1}&(A\boxtimes B)\boxtimes(C\boxtimes D)\ar[r]^{\overline{\alpha}}&A\boxtimes(B\boxtimes(C\boxtimes D))\\
(A\boxtimes(B\boxtimes C))\boxtimes D\ar[rr]^-{\overline{\alpha}}&&A\boxtimes((B\boxtimes C)\boxtimes D)\ar[u]_{1\boxtimes\overline{\alpha}}}
\]
follows from commutativity of the coherence diagram of $T(((A\otimes B)\otimes C)\otimes D)$, the definition and naturality of $\overline{\alpha}$, the fact that $q_{A,B}=a\boxtimes b$, and the hypothesis that $(q_{A,B}\boxtimes c)\boxtimes d\cong q_{A,B}\boxtimes q_{C,D}$ is epic.

\textit{Unitariness.} For a $\mT$-algebra $(A,a)$, the composite $\cC$-morphism 
\[
\xymatrix{TE\otimes A\ar[r]^-{1\otimes\eta}&TE\otimes TA\ar[r]^-{\kappa}&T(E\otimes A)\ar[r]^-{T\lambda}&TA\ar[r]^{a}&A}
\]
is a bimorphism: on one hand, we have
\begin{align*}
a\cdot T(a\cdot T\lambda_A\cdot\kappa_{E,A}&\cdot(1_{TE}\otimes\eta_A))\cdot\kappa_{TE,A}\\
&=a\cdot T\lambda_A\cdot\mu_{E\otimes A}\cdot T\kappa_{E,A}\cdot\kappa_{TE,TA}\cdot(1_{TTE}\otimes T\eta_A)\\
&=a\cdot T\lambda_A\cdot\kappa_{E,A}\cdot(\mu_{E}\otimes\mu_{A})\cdot(1_{TTE}\otimes T\eta_A)\\
&=a\cdot T\lambda_A\cdot\kappa_{E,A}\cdot(\mu_{E}\otimes 1_{TA})\ ,
\end{align*}
and on the other hand,
\begin{align*}
a\cdot T\lambda_A\cdot\kappa_{E,A}&\cdot(1_{TE}\otimes\eta_A)\cdot(\mu_{E}\otimes a)\\
&=a\cdot T\lambda_A\cdot\kappa_{E,A}\cdot(1_{TE}\otimes Ta\cdot\eta_{TA})\cdot(\mu_{E}\otimes 1_{TA})\\
&=a\cdot Ta\cdot T\lambda_{TA}\cdot\kappa_{E,TA}\cdot(1_{TE}\otimes\eta_{TA})\cdot(\mu_{E}\otimes 1_{TA})\\
&=a\cdot \mu_{A}\cdot T(T\lambda_A\cdot\kappa_{E,A}\cdot(\eta_E\otimes 1_{TA}))\cdot\kappa_{E,TA}\cdot(1_{TE}\otimes\eta_{TA})\cdot(\mu_{E}\otimes 1_{TA})\\
&=a\cdot T\lambda_A\cdot\mu_{E\otimes A}\cdot T\kappa_{E,A}\cdot\kappa_{TE,TA}\cdot(T\eta_E\otimes \eta_{TA})\cdot(\mu_{E}\otimes 1_{TA})\\
&=a\cdot T\lambda_A\cdot\kappa_{E,A}\cdot(\mu_{E}\otimes\mu_{A})\cdot(T\eta_E\otimes \eta_{TA})\cdot(\mu_{E}\otimes 1_{TA})\\
&=a\cdot T\lambda_A\cdot\kappa_{E,A}\cdot(\mu_{E}\otimes 1_{TA})\ .
\end{align*}
By Lemma \ref{lem:bim}, there is therefore a unique $\cC^\mT$-morphism $\overline{\lambda}_A$ that makes the diagram
\[
\xymatrix@C=6.5em{T(TE\otimes A)\ar[d]_{q}\ar[r]^-{T(a\cdot T\lambda\cdot\kappa\cdot(1\otimes\eta))}&TA\ar[d]^{a}\\
TE\boxtimes A\ar[r]^-{\overline{\lambda}}&A}
\]
commute. Its inverse is the $\cC^\mT$-morphism induced by $l:=q_{TE,A}\cdot T(\eta_E\otimes 1_A)\cdot T\lambda_A^\inv:TA\to TE\boxtimes A$. Indeed, since $a$ is a coequalizer in $\cC^\mT$ of $(Ta,\mu_A)$, and
\begin{align*}
(q_{TE,A}\cdot T(\eta_E\otimes 1_A)&\cdot T\lambda_A^\inv)\cdot Ta\\
&=q_{TE,A}\cdot T(1_{TE}\otimes a)\cdot T(\eta_E\otimes 1_{TA})\cdot T\lambda_{TA}^\inv\\
&=q_{TE,A}\cdot T(\mu_{E}\otimes a)\cdot T(T\eta_{E}\otimes1_{TA})\cdot T(\eta_{E}\otimes 1_{TA})\cdot T\lambda_{TA}^\inv\\
&=q_{TE,A}\cdot\mu_{TE\otimes A}\cdot T\kappa_{TE,A}\cdot T(T\eta_{E}\otimes1_{TA})\cdot T(\eta_{E}\otimes 1_{TA})\cdot T\lambda_{TA}^\inv\\
&=q_{TE,A}\cdot\mu_{TE\otimes A}\cdot TT(\eta_{E}\otimes1_{A})\cdot T\kappa_{E,A} \cdot T(\eta_{E}\otimes 1_{TA})\cdot T\lambda_{TA}^\inv\\
&=q_{TE,A}\cdot\mu_{TE\otimes A}\cdot TT(\eta_{E}\otimes1_{A})\cdot TT\lambda_{A}^\inv\\
&=(q_{TE,A}\cdot T(\eta_{E}\otimes1_{TA})\cdot T\lambda_{TA}^\inv)\cdot\mu_{A}\ ,
\end{align*}
there is a unique induced $\cC^\mT$-morphism $A\to TE\boxtimes A$ (given by $l\cdot\eta_A$). Using this computation, we can also compose $l$ with $T(a\cdot T\lambda_{A}\cdot\kappa_{TE,A}\cdot(1_{TE}\otimes\eta_A))$ to obtain
\begin{align*}
q_{TE,A}\cdot T(\eta_E\otimes 1_A)&\cdot T\lambda_A^\inv\cdot T(a\cdot T\lambda_{A}\cdot\kappa_{TE,A}\cdot(1_{TE}\otimes\eta_A))\\
&=q_{TE,A}\cdot T(\eta_E\otimes 1_A)\cdot T\lambda_A^\inv\cdot T\lambda_{A}\cdot\mu_{TE\otimes A}\cdot T\kappa_{TE,A}\cdot T(1_{TE}\otimes\eta_A)\\
&=q_{TE,A}\cdot\mu_{TE\otimes A}\cdot T\kappa_{TE,A}\cdot T(T\eta_{E}\otimes\eta_A)\\
&=q_{TE,A}\cdot T(\mu_{E}\otimes a)\cdot T(T\eta_{E}\otimes\eta_A)\\
&=q_{TE,A}\ .
\end{align*}
Since $q_{TE,A}$ is epic, we have $l\cdot\eta_A\cdot\overline{\lambda}_A=1_{TE\boxtimes A}$. With
\[
a\cdot(T(a\cdot T\lambda_{A}\cdot\kappa_{TE,A}\cdot(1_{TE}\otimes\eta_A)))\cdot (T(\eta_E\otimes 1_A)\cdot T\lambda_A^\inv)=a\ ,
\]
we obtain similarly $\overline{\lambda}_A\cdot l\cdot\eta_A=1_{A}$. Hence, $\overline{\lambda}_A$ is an isomorphism in $\cC^\mT$. Naturality of $\overline{\lambda}$ follows from a standard diagram chase involving the defining diagrams of $\overline{\lambda}_A$, $\overline{\lambda}_B$, the facts that $1_{TE}\boxtimes f$ and $f$ are $\cC^\mT$-morphisms, and that $q_{TE,A}$ is epic. The natural isomorphism $\overline{\rho}_A:A\boxtimes TE\to A$ is obtained symmetrically via the defining diagram
\[
\xymatrix@C=6.5em{T(A\otimes TE)\ar[d]_{q}\ar[r]^-{T(a\cdot T\rho\cdot\kappa\cdot(\eta\otimes 1))}&TA\ar[d]^{a}\\
A\boxtimes TE\ar[r]^-{\overline{\rho}}&A\rlap{\ .}}
\]
Commutativity of
\[
\xymatrix@C=1em{(A\boxtimes TE)\boxtimes B\ar[dr]_{\overline{\rho}\boxtimes 1}\ar[rr]^-{\overline{\alpha}}&&A\boxtimes(TE\boxtimes B)\ar[dl]^{1\boxtimes\overline{\lambda}}\\
&A\boxtimes B&}
\]
follows again by a standard diagram chase involving the defining diagrams of the three given morphisms, and the fact that $q_{A,TE}\boxtimes b$ is epic.

\textit{Symmetry.} By using symmetry of the monoidal monad $(\mT,\kappa)$, one obtains the existence of a family of $\cC^\mT$-morphisms $\overline{\sigma}_{A,B}:A\boxtimes B\to B\boxtimes A$ via the diagram
\[
\xymatrix{T(A\otimes B)\ar[r]^-{T\sigma}\ar[d]_{q}&T(B\otimes A)\ar[d]^q\\
A\boxtimes B\ar[r]^-{\overline{\sigma}}&B\boxtimes A\rlap{\ .}}
\]
Naturality of $\overline{\sigma}$, as well as commutativity of the symmetry diagrams for a symmetric monoidal category then follows by straightforward diagrammatic arguments.
\end{proof}


The following result is a more memorable version of Theorem~\ref{thm:BoxtimesMonoidal}.

\begin{cor}\label{cor:BoxtimesMonoidal}
Let $(\mT,\kappa)$ be a monoidal monad on a monoidal category $(\cC,\otimes,E)$ such that $\cC^\mT$ has tensors. If $A\boxtimes(-)$ and $(-)\boxtimes B$ preserve reflexive coequalizer diagrams in $\cC^\mT$ (for all $\mT$-algebras $A$, $B$), then $(\cC^\mT,\boxtimes,TE)$ is a monoidal category whose structure morphisms are induced by those of $(\cC,\otimes,E)$.

Moreover, if $(\cC,\otimes,E)$ and $(\mT,\kappa)$ are symmetric monoidal, then so is $(\cC^\mT,\boxtimes,TE)$.
\end{cor}

\begin{proof}
The statement follows directly from Theorem~\ref{thm:BoxtimesMonoidal}. Indeed, all $q_{A,B}$ and $\mT$-algebras structures are reflexive coequalizers, so Proposition~\ref{prop:TensorReflexiveCoeq} yields that $q_{A,B}\boxtimes c$,  $a\boxtimes q_{B,C}$ and $q_{A,B}\boxtimes q_{C,D}$ are coequalizers of their respective diagrams.
\end{proof}

%
%

In the case where $(\cC,\otimes,E)$ is closed symmetric monoidal, the closed structure can be lifted to $\cC^\mT$ provided that $\cC$ has equalizers and reflexive coequalizers. 

\begin{cor}\label{cor:thm:BoxtimesMonoidal}
Let $(\mT,\kappa)$ be a monoidal monad on a closed symmetric monoidal category $(\cC,\otimes,E)$ with equalizers and such that $\cC^\mT$ has tensors. Then $(\cC^\mT,\boxtimes,TE)$ is a monoidal category whose structure morphisms are induced by those of $(\cC,\otimes,E)$.
\end{cor}

\begin{proof}
The equalizer hypothesis allows us to apply the construction of an internal hom in $\cC^\mT$ described in \cite[Theorem~2.2]{Koc:71a} that makes $A\boxtimes(-)$ left adjoint for any $A\in\cC^\mT$ (see also the proof of  \cite[Proposition~1.3]{Koc:71b}). In particular, $A\boxtimes(-)$ preserves reflexive coequalizers, and Corollary~\ref{cor:BoxtimesMonoidal} applies.
\end{proof}

\begin{examples}\label{ex:BoxtimesMonoidal}
\item\label{ex:BoxtimesMonoidal:Trivial} In the case where $(\mT,\kappa)=(\mId,1)$ is the identity monoidal monad on a monoidal category $(\cC,\otimes,E)$, Theorem~\ref{thm:BoxtimesMonoidal} reduces to a tautology, namely that $(\cC,\otimes,E)$ is monoidal.
\item For the free abelian group cartesian monad $(\mAb,\kappa)$ on $\Set$ (Example~\ref{ex:MonoidalMonads}\ref{ex:MonoidalMonads:Ab}), Corollary~\ref{cor:thm:BoxtimesMonoidal} describes the usual tensor product over $\mathbb{Z}$ of the category $\Set^\mAb\cong\AbGrp$ of abelian groups.
\item\label{ex:BoxtimesMonoidal:P} For the monoidal powerset monad $(\mP,\iota)$ on $\Set$ (Example~\ref{ex:MonoidalMonads}\ref{ex:MonoidalMonads:P}), Corollary~\ref{cor:thm:BoxtimesMonoidal} yields the classical tensor product on $\Set^\mP\cong\Sup$ (see \cite{Shm:74}).
\end{examples}

\subsection{Alternative hypotheses}
In the case where  $(\cC,\otimes,E)$ is not closed symmetric monoidal (see Corollary~\ref{cor:thm:BoxtimesMonoidal}), it might be delicate to verify the more general hypotheses of Theorem~\ref{thm:BoxtimesMonoidal} or even of Corollary~\ref{cor:BoxtimesMonoidal} since each of these involve the \emph{induced} tensor $(-\boxtimes-)$. Theorem~\ref{thm:BoxtimesMonoidal:practical} below presents a situation where the hypotheses can be directly tested on the original data $\mT$ and $(\cC,\otimes,E)$. The proof hinges on how certain coequalizers can be tensored.

\begin{prop}\label{prop:BoxtimesMonoidal:practical}
Let $(\mT,\kappa)$ be a monoidal monad on a monoidal category $(\cC,\otimes,E)$, and consider coequalizer diagrams 
\[
\xymatrix{TZ\ar[r]<0.7ex>^-{\mu\cdot Tr}\ar[r]<-0.7ex>_-{Ts}&TY\ar[r]^-{p}&X}\dand
\xymatrix{TZ'\ar[r]<0.7ex>^-{\mu\cdot Tr'}\ar[r]<-0.7ex>_-{Ts'}&TY'\ar[r]^-{p'}&X'}
\]
in $\cC^\mT$  (here, $(X,x)$, $(X',x')$ are $\mT$-algebras, and the objects of the form $TA$ are equipped with their free structure $\mu_A$). Suppose that $T(-\otimes -):\cC^\mT\times\cC^\mT\to\cC^\mT$ preserves the coequalizer diagram of $(p,p')$, that $T(Tp\otimes Tp')$ is an epimorphism, and that the coequalizer $q_{X,X'}:T(X\otimes X')\to X\boxtimes X$ exists.

If the coequalizer of $\mu_{Y\otimes Y'}\cdot T(\kappa_{Y,Y'}\cdot(r\otimes r'))$ and $T(s\otimes s')$ exists $\cC^\mT$, then it can be given by $q_{X,X'}\cdot T(p\cdot\eta_{Y}\otimes p'\cdot\eta_{Y'})$:
\[
\xymatrix@C=5em{T(Z\otimes Z')\ar[r]<0.7ex>^-{\mu\cdot T(\kappa\cdot(r\otimes r'))}\ar[r]<-0.7ex>_-{T(s\otimes s')}&T(Y\otimes Y')\ar[r]^-{q\cdot T(p\cdot\eta\otimes p'\cdot\eta)}&X\boxtimes X'\rlap{\ .}}
\]
\end{prop}

\begin{proof}
By hypothesis, there are coequalizer diagrams
\[
\xymatrix@C=7em{T(TZ\otimes TZ')\ar[r]<0.7ex>^-{T((\mu\cdot Tr)\otimes(\mu\cdot Tr'))}\ar[r]<-0.7ex>_-{T(Ts\otimes Ts')}&T(TY\otimes TY')\ar[r]^-{T(p\otimes p')}&T(X\otimes X')\rlap{\ .}}
\]
and
\[
\xymatrix@C=6em{T(Z\otimes Z')\ar[r]<0.7ex>^-{\mu\cdot T(\kappa\cdot(r\otimes r'))}\ar[r]<-0.7ex>_-{T(s\otimes s')}&T(Y\otimes Y')\ar[r]^-{t}&X\circledast X'}
\]
in $\cC^\mT$. We proceed to show that $X\circledast X'$ and $X\boxtimes X'$ are isomorphic. Since the diagram
\[
\xymatrix@C=7em{T(TZ\otimes TZ')\ar[r]<0.7ex>^-{T((\mu\cdot Tr)\otimes(\mu\cdot Tr'))}\ar[r]<-0.7ex>_-{T(Ts\otimes Ts')}&T(TY\otimes TY')\ar[r]^-{t\cdot\mu\cdot T\kappa}&X\circledast X'}
\]
commutes, there exists a unique $\cC^\mT$-morphism $\overline{t}:T(X\otimes X')\to X\circledast X'$ such that
\begin{equation}\label{eq:prop:TrimEq}
\overline{t}\cdot T(p\otimes p')=t\cdot\mu_{Y\otimes Y'}\cdot T\kappa_{Y,Y'}\ .
\end{equation}
We can therefore consider the diagram 
\[
\xymatrix@C=4em{T(TTY\otimes TTY')\ar[r]^-{T(Tp\otimes Tp')}\ar[d]<0.3em>^{\mu\cdot T\kappa}\ar[d]<-0.3em>_{T(\mu\otimes\mu)}&T(TX\otimes TX')\ar[d]<-0.3em>_{\mu\cdot T\kappa}\ar[d]<0.3em>^{T(x\otimes x')}\\
T(TY\otimes TY')\ar[r]^-{T(p\otimes p')}\ar[d]_{\mu\cdot T\kappa}&T(X\otimes X')\ar[d]^{\overline{t}}\\
T(Y\otimes Y')\ar[r]^-{t}&X\circledast X'}
\]
in which the inner- and outer-upper squares commute, as does the lower one; Proposition~\ref{prop:TensorFree} moreover implies that the two left vertical arrows of the large square can be identified. By hypothesis, $T(Tp\otimes Tp')$ is an epimorphism, so the vertical $\cC^\mT$-morphism $T(TX\otimes TX')\to X\circledast X'$ that makes the large square commute is unique, that is,
\[
\xymatrix@C=4.5em{T(TX\otimes TX')\ar[r]<0.7ex>^-{\mu\cdot T\kappa}\ar[r]<-0.7ex>_-{T(x\otimes x')}&T(X\otimes X')\ar[r]^{\overline{t}}&X\circledast X'}
\]
commutes. Thus, the universal property of $q_{X,X'}$ yields a unique $\cC^\mT$-morphism $\overline{\overline{t}}:X\boxtimes X'\to X\circledast X'$ such that
\[
\overline{\overline{t}}\cdot q_{X,X'}=\overline{t}\ .
\]
Let us verify now that $\overline{\overline{t}}$ is an isomorphism. With \eqref{eq:prop:TrimEq}, one obtains
\begin{equation}\label{eq:prop:PenUltCentralArg}
\overline{\overline{t}}\cdot q_{X,X'}\cdot T(p\otimes p')=t\cdot\mu_{Y\otimes Y'}\cdot T\kappa_{Y,Y'}\ ,
\end{equation}
and therefore
\[
t=\overline{\overline{t}}\cdot q_{X,X'}\cdot T(p\otimes p')\cdot T(\eta_{Y}\otimes\eta_{Y'})
=\overline{\overline{t}}_{X,X'}\cdot q_{X,X'}\cdot T(p\cdot\eta_{Y}\otimes p'\cdot\eta_{Y'})\ .
\]
The equality
\begin{equation}\label{eq:prop:ConclCentralArg}
q_{X,X'}\cdot T(p\cdot\eta_{Y}\otimes p'\cdot\eta_{Y'})\cdot\mu_{Y\otimes Y'}\cdot T\kappa_{Y,Y'}=q_{X,X'}\cdot T(p\otimes p')\ ,
\end{equation}
with Lemma \ref{lem:CoeqSimple} shows that
\[
\xymatrix@C=5em{T(Z\otimes Z')\ar[r]<0.7ex>^-{\mu\cdot T(\kappa\cdot(r\otimes r'))}\ar[r]<-0.7ex>_-{T(s\otimes s')}&T(Y\otimes Y')\ar[r]^-{q\cdot T(p\cdot\eta\otimes p'\cdot\eta)}&X\boxtimes X'}
\]
commutes. The universal property of $t$ then yields a unique $\cC^\mT$-morphism $u:X\circledast X'\to X\boxtimes X'$ such that
\[
q_{X,X'}\cdot T(p\cdot\eta_{Y}\otimes p'\cdot\eta_{Y'})=u\cdot t\ ,
\]
and we can consider the following commutative diagram:
\[
\xymatrix@C=3.5em{T(TY\otimes TY')\ar[r]^-{T(p\otimes p')}\ar[d]_{\mu\cdot T\kappa}&T(X\otimes X')\ar[r]^-{q}\ar[dr]^{\overline{t}}&X\boxtimes X'\ar[d]^{\overline{\overline{t}}}\\
T(Y\otimes Y')\ar[rr]^{t}\ar[d]_{T(\eta\otimes\eta)}&&X\circledast X'\ar[d]^{u}\\
T(TY\otimes TY')\ar[r]^-{T(p\otimes p')}&T(X\otimes X')\ar[r]^-{q}&X\boxtimes X'\rlap{\ .}}
\]
By using \eqref{eq:prop:ConclCentralArg}, one then observes
\[
q_{X,X'}\cdot T(p\otimes p')=u\cdot\overline{\overline{t}}\cdot q_{X,X'}\cdot T(p\otimes p')\ ,
\]
so $u\cdot\overline{\overline{t}}=1_{X\boxtimes X'}$ because both $q_{X,X'}$ and $T(p\otimes p')$ are epimorphisms. By exchanging the displayed upper and lower diagrams, one obtains similarly that $\overline{\overline{t}}\cdot u=1_{X\circledast X'}$, and can conclude that $\overline{\overline{t}}$ is an isomorphism with inverse $u$.
\end{proof}

\begin{cor}\label{cor:TriQ}
Let $(\mT,\kappa)$ be a monoidal monad on a monoidal category $(\cC,\otimes,E)$ with reflexive coequalizers. If $T(X\otimes-)$ and $T(-\otimes Y)$ preserve reflexive coequalizers in $\cC$ (for all $X,Y\in\ob\cC$), then
\[
\xymatrix@C=4.8em{T((TA\otimes TB)\otimes TC)\ar[r]<0.7ex>^-{\mu\cdot T(\kappa\cdot(\kappa\otimes 1))}\ar[r]<-0.7ex>_-{T((a\otimes b)\otimes c)}&T((A\otimes B)\otimes C)\ar[r]^-{q\cdot T(q\cdot\eta\otimes 1)}&(A\boxtimes B)\boxtimes C}
\]
and
\[
\xymatrix@C=4.8em{T(TA\otimes(TB\otimes TC))\ar[r]<0.7ex>^-{\mu\cdot T(\kappa\cdot(1\otimes\kappa))}\ar[r]<-0.7ex>_-{T(a\otimes(b\otimes c))}&T(A\otimes(B\otimes C))\ar[r]^-{q\cdot T(1\otimes q\cdot\eta)}&A\boxtimes(B\boxtimes C)}
\]
are coequalizer diagrams in $\cC^\mT$ (for all $\mT$-algebras $(A,a)$, $(B,b)$ and $(C,c)$).
\end{cor}

\begin{proof}
We prove the statement for the first diagram, the proof for the second following similarly. For this, we only need to verify that the hypotheses of Proposition~\ref{prop:BoxtimesMonoidal:practical} are verified for the reflexive coequalizer diagrams
\[
\xymatrix@C=3em{T(TA\otimes TB)\ar[r]<0.7ex>^-{\mu\cdot T\kappa}\ar[r]<-0.7ex>_-{T(a\otimes b)}&T(A\otimes B)\ar[r]^-{q}&A\boxtimes B}\dand
\xymatrix@C=3em{TTC\ar[r]<0.7ex>^-{\mu}\ar[r]<-0.7ex>_-{Tc}&TC\ar[r]^-{c}&C\rlap{\ .}}
\]
Since $T\cong T(E\otimes -)$ preserves reflexive coequalizers, these lift to $\cC^\mT$ by Proposition~\ref{prop:CreatReflCoeq}, and $T(q_{A,B}\otimes 1_C)$ is therefore a coequalizer in $\cC^\mT$; thus, $T(q_{A,B}\otimes c)$ is one too by Proposition~\ref{prop:TensorReflexiveCoeq}. Similarly, $T(Tq_{A,B}\otimes Tc)$ is a reflexive coequalizer in $\cC^\mT$ and consequently an epimorphism. Finally, the coequalizer of
\[
\mu_{(A\otimes B)\otimes C}\cdot T(\kappa_{A\otimes B,C}\cdot(\kappa_{A\otimes B}\otimes 1_C))\dand T(a\otimes b)\otimes Tc
\]
exists in $\cC^\mT$ because it is a reflexive pair (split by $T((\eta_A\otimes\eta_B)\otimes \eta_C)$).
\end{proof}

\begin{cor}\label{cor:QuaQ}
Let $(\mT,\kappa)$ be a monoidal monad on a monoidal category $(\cC,\otimes,E)$ with reflexive coequalizers. If $T(X\otimes-)$ and $T(-\otimes Y)$ preserve reflexive coequalizers in $\cC$ (for all $X,Y\in\ob\cC$), then
\[
\xymatrix@C=5em{T((A\otimes B)\otimes(C\otimes D))\ar[r]^-{q\cdot T(q\cdot\eta\otimes q\cdot\eta)}&(A\boxtimes B)\boxtimes(C\boxtimes D)}
\]
is an epimorphism in $\cC^\mT$ (for all $\mT$-algebras $(A,a)$, $(B,b)$, $(C,c)$, and $(D,d)$).
\end{cor}

\begin{proof}
The given morphism is in fact a reflexive coequalizer in $\cC^\mT$, obtained by applying Proposition~\ref{prop:BoxtimesMonoidal:practical} to the coequalizer diagrams of $q_{A,B}$ and $q_{C,D}$.
\end{proof}

\begin{thm}\label{thm:BoxtimesMonoidal:practical}
Let $(\mT,\kappa)$ be a monoidal monad on a monoidal category $(\cC,\otimes,E)$ with reflexive coequalizers. If $T(X\otimes-)$ and $T(-\otimes Y)$ preserve reflexive coequalizers in $\cC$ (for all $X,Y\in\ob\cC$), then $(\cC^\mT,\boxtimes,TE)$ is a monoidal category whose structure morphisms are induced by those of $(\cC,\otimes,E)$.

Moreover, if $(\cC,\otimes,E)$ and $(\mT,\kappa)$ are symmetric monoidal, then so is $(\cC^\mT,\boxtimes,TE)$.
\end{thm}

\begin{proof}
We verify the hypotheses of Theorem~\ref{thm:BoxtimesMonoidal}. By Proposition~\ref{prop:CreatReflCoeq}, $q_{A,B}$ exists and is therefore is a reflexive coequalizer of $(\mu_{A\otimes B}\cdot T\kappa_{A,B},T(a\otimes b))$ in $\cC^\mT$. With the convention of \ref{conv:Tensor}, the first coequalizer diagram in Corollary~\ref{cor:TriQ} is in fact
\[
\xymatrix@C=4.8em{T(TA\otimes TB)\boxtimes TTC\ar[r]<0.7ex>^-{\mu\cdot T\kappa\boxtimes \mu}\ar[r]<-0.7ex>_-{T(a\otimes b)\boxtimes Tc}&T(A\otimes B)\boxtimes TC\ar[r]^-{q\boxtimes c}&(A\boxtimes B)\boxtimes C\rlap{\ .}}
\]
Indeed, we already know that $T((a\otimes b)\otimes c)=T(a\otimes b)\boxtimes Tc$; moreover, with $q_{TA,TB}=\mu_{A\otimes B}\cdot T\kappa_{A,B}$ in the diagram
\[
\xymatrix@C=4.8em{T(T(TA\otimes TB)\otimes TTC)\ar[d]_{\mu\cdot T\kappa}\ar[r]^-{T(\mu\cdot T\kappa\otimes\mu)}&T(T(A\otimes B)\otimes TC)\ar[d]^{\mu\cdot T\kappa}\\
T((TA\otimes TB)\otimes TC)\ar[r]^-{\mu\cdot T(\kappa\cdot(\kappa\otimes 1))}&T((A\otimes B)\otimes C)\rlap{\ ,}}
\]
we see that $\mu_{(A\otimes B)\otimes C}\cdot T(\kappa_{A\otimes B,C}\cdot(\kappa_{A,B}\otimes 1_{TC}))=\mu_{A\otimes B}\cdot T\kappa_{A,B}\boxtimes\mu_C$; finally, since $q_{A,B}$ is a $\cC^\mT$-morphism and $q_{A\boxtimes B,C}=(a\bowtie b)\boxtimes c$, we have by functoriality of $(-\boxtimes-)$:
\begin{align*}
q_{A\boxtimes B,C}\cdot T(q_{A,B}\cdot\eta_{A\otimes B}\otimes 1_C)&=((a\bowtie b)\boxtimes c)\cdot (T(q_{A,B}\cdot\eta_{A\otimes B})\boxtimes 1_{TC})\\
&=(q_{A,B}\boxtimes c)\cdot ((\mu_{A\otimes B}\cdot T\eta_{A\otimes B})\boxtimes 1_{TC})=q_{A,B}\boxtimes c\ .
\end{align*}
By symmetry, the second coequalizer diagram of  Corollary~\ref{cor:TriQ} shows that $a\boxtimes q_{B,C}$ is the coequalizer of $(\mu_A\boxtimes\mu_{B\otimes C}\cdot T\kappa_{B,C},Ta\boxtimes T(b\otimes c))$. By using that $q_{A,B}=a\boxtimes b$ and $q_{C,D}=c\boxtimes d$ are $\cC^\mT$-morphisms in Corollary~\ref{cor:QuaQ}, we obtain that
\[
q_{A,B}\boxtimes q_{C,D}=q_{A\boxtimes B,C\boxtimes D}\cdot T(q_{A,B}\cdot\eta_{A\otimes B}\otimes q_{C,D}\cdot\eta_{C\otimes D})
\]
is an epimorphism. Hence, Theorem~\ref{thm:BoxtimesMonoidal} applies. 
\end{proof}

\begin{ex}\label{ex:BoxtimesMonoidal:practical}
For the monoidal category $(\cC,+,\varnothing)$ of Example~\ref{ex:MonoidalMonads}\ref{ex:MonoidalMonads:coprod}, we note that the functor $X+(-)$ preserves coequalizers for all $X\in\ob\cC$. In general, $\cC$ is not monoidal closed (for example, if $\cC=\Set$ and $X$ is a non-empty set, the functor $X+(-)$ is not left adjoint, as it does not preserve coproducts), so Corollary~\ref{cor:thm:BoxtimesMonoidal} does not apply. Nevertheless, if $\cC^\mT$ has tensors, then $(-\boxtimes-)$ turns out to be the binary coproduct in $\cC^\mT$ (see \cite[Proposition~2]{Lin:69}). In this case, $(\cC^\mT,\boxtimes,T\varnothing)$ is monoidal because of the universal property of coproducts; alternatively, $(-\boxtimes-)$ is left adjoint to the diagonal functor, so Corollary~\ref{cor:BoxtimesMonoidal} applies. 

In the case where $\cC$ has reflexive coequalizers and the monad functor $T$ preserves them (so that $\cC^\mT$ has reflexive coequalizers by Proposition~\ref{prop:CreatReflCoeq}), Theorem~\ref{thm:BoxtimesMonoidal:practical} states that $(\cC^\mT,\boxtimes,T\varnothing)$ is monoidal, a result that essentially summarizes the previous discussion.
\end{ex}


\subsection{Monoidal monad morphisms and Eilenberg--Moore categories}
Once a monoidal structure on the Eilenberg--Moore category has been established, it is not surprising that monoidal monad morphisms induce monoidal functors, although these need not be strict as in  Proposition~\ref{prop:MonoidalMorphismLift}.

\begin{prop}\label{prop:MonadMorphMonoidal}
Let $(\cC,\otimes,E)$ be a monoidal category, and $(\mS,\iota)$, $(\mT,\kappa)$ monoidal monads on $\cC$. Suppose that $\cC^\mT$ and $\cC^\mS$ both have tensors that make $(\cC^\mT,\boxtimes,TE)$ and $(\cC^\mS,\circledast,SE)$ monoidal categories with structures induced by those of $(\cC,\otimes,E)$. If $\phi:(\mS,\iota)\to(\mT,\kappa)$ is a monoidal monad morphism, then the induced functor
\[
\cC^\phi:\cC^\mT\to\cC^\mS
\]
(that commutes with the forgetful functors to $\cC$) is itself monoidal. 
\end{prop}

\begin{proof}
Set $\mS=(S,\nu,\delta)$ and $\mT=(T,\mu,\eta)$. Recall that the algebraic functor $\cC^\phi$ sends a $\mT$-algebra $(A,a)$ to the $\mS$-algebra $(A,a\cdot\phi_A)$ and is identical on morphisms. For $\mT$-algebras $(A,a)$ and $(B,b)$, one has $\phi_{A\otimes B}\cdot\nu_{A\otimes B}=\mu_{A\otimes B}\cdot\phi_{T(A\otimes B)}\cdot S\phi_{A\otimes B}$, that is, $\phi_{A\otimes B}:S(A\otimes B)\to T(A\otimes B)$ is a $\cC^\mS$-morphism. Hence, the inner- and outer-left squares in the diagram
\[
\xymatrix@C=5.5em{S(SA\otimes SB)\ar[r]<0.7ex>^-{\nu\cdot S\iota}\ar[r]<-0.7ex>_-{S((a\cdot\phi)\otimes(b\cdot\phi))}\ar[d]_{\phi\cdot S(\phi\otimes\phi)}&S(A\otimes B)\ar[r]^{p}\ar[d]^{\phi}&A\circledast B\ar@{.>}[d]^{\overline{\phi}}\\
T(TA\otimes TB)\ar[r]<-0.7ex>_-{\mu\cdot T\kappa}\ar[r]<0.7ex>^-{T(a\otimes b)}&T(A\otimes B)\ar[r]^{q}&A\boxtimes B}
\]
commute in $\cC^\mS$. If $p_{A,B}$ denotes the coequalizer of the upper row, there is consequently a unique $\cC^\mS$-morphism $\overline{\phi}_{A,B}:(A\circledast B,(a\cdot\phi_A)\bowtie(b\cdot\phi_B))\to(A\boxtimes B,(a\bowtie b)\cdot\phi_{A\boxtimes B})$ that makes the square on the right commute. These $\overline{\phi}_{A,B}$ (for $A,B\in\ob\cC^\mT$) are easily seen to be natural in $A$ and $B$. 

A standard diagram chase involving defining diagrams of the induced structures yields commutativity of
\[
\xymatrix{(A\circledast B)\circledast C\ar[r]^{\overline{\phi}\circledast1}\ar[d]_{\overline{\alpha}}&(A\boxtimes B)\circledast C\ar[r]^{\overline{\phi}}&(A\boxtimes B)\boxtimes C\ar[d]^{\overline{\alpha}}\\
A\circledast(B\circledast C)\ar[r]^{1\circledast\overline{\phi}}&A\circledast(B\boxtimes C)\ar[r]^{\overline{\phi}}&A\boxtimes(B\boxtimes C)\rlap{\ .}}
\]
Similarly, commutativity of 
\[
\xymatrix{SE\circledast A\ar[r]^-{\phi\circledast1}\ar[dr]_{\overline{\lambda}}&TE\circledast A\ar[r]^{\overline{\phi}}&TE\boxtimes A\ar[dl]^{\overline{\lambda}}\\
&A&}
\qquad\xymatrix{\mbox{}\ar@{}[d]|{\textstyle{\text{and}}}\\\mbox{}}\qquad
\xymatrix{A\circledast SE\ar[r]^-{1\circledast\phi}\ar[dr]_{\overline{\rho}}&A\circledast TE\ar[r]^{\overline{\phi}}&A\boxtimes TE\ar[dl]^{\overline{\rho}}\\
&A&}
\]
is routinely verified. Hence, the functor $\cC^\phi$ with $\overline{\phi}:\cC^\phi(-)\circledast\cC^\phi(-)\to\cC^\phi(-\boxtimes-)$ and $\phi_E:SE\to TE$ is monoidal.
\end{proof}

\begin{ex}\label{ex:MonoidalMonMorph}
The unit monad morphism $\eta:\mId\to\mT$ of a monoidal monad $\mT$ (Example~\ref{ex:MonoidalMonadsMorph}) induces the forgetful functor $\cC^\eta:\cC^\mT\to\cC^\mId\cong\cC$ that is therefore monoidal (whenever $(\cC^\mT,\otimes,TE)$ is monoidal with structures morphisms induced by $(\cC,\otimes,E)$).
\end{ex}

\begin{rem}
Example~\ref{ex:MonoidalMonMorph} shows that in the context of Theorem~\ref{thm:BoxtimesMonoidal}, the forgetful functor $G^\mT:\cC^\mT\to\cC$ is monoidal. Standard diagram chases also show that the left adjoint $F^\mT:\cC\to\cC^\mT$ is strong monoidal with respect to $\overline{\kappa}:F^\mT(-)\boxtimes F^\mT(-)\to F^\mT(-\otimes-)$ and $1_{TE}:TE\to TE$ (where $\overline{\kappa}_{X,Y}:(TX\boxtimes TY,\mu_X\bowtie\mu_Y)\to(T(X\otimes Y),\mu_{X\otimes Y})$ denotes the $\cC^\mT$-morphism induced by the bimorphism $\kappa_{X,Y}$ for all $X,Y\in\ob\cC$). Hence, the adjunction $F^\mT\dashv G^\mT:\cC^\mT\to\cC$ is monoidal, mirroring the closed case studied in \cite{Koc:71a} in relation with commutative monads. 
\end{rem}

\section{Actions}\label{sec:Actions}

\subsection{Monoids}\label{Monoids:Monoidal}
A \df{monoid} in a monoidal category $\cC$ is a $\cC$-object $M$ together with two morphisms
\[
m:M\otimes M\to M\ ,\quad e:E\to M
\]
such that the diagrams
\begin{equation*}
\xymatrix@C=2.5em{(M\otimes M)\otimes M\ar[d]_-{m\otimes 1}\ar[r]^-{\alpha}&M\otimes(M\otimes M)\ar[r]^-{1\otimes m}&M\otimes M\ar[d]^-{m}\\
M\otimes M\ar[rr]^-{m}&&M}\quad
\xymatrix@C=2.5em{E\otimes M\ar[dr]_-{\lambda}\ar[r]^-{e\otimes 1}&M\otimes M\ar[d]^-{m}&M\otimes E\ar[dl]^-{\rho}\ar[l]_-{1\otimes e}\\
&M&}
\end{equation*}
commute. A \df{homomorphism of monoids} $f:(N,n,d)\to(M,m,e)$ is a $\cC$-morphism $f:N\to M$ such that the diagrams
\[
\xymatrix@C=35pt{N\otimes N\ar[d]_-{n}\ar[r]^-{f\otimes f}&M\otimes M\ar[d]^-{m}\\
M\ar[r]^-{f}&M}\qquad\qquad
\xymatrix@C=25pt{&E\ar[dl]_-{d}\ar[dr]^-{e}&\\
N\ar[rr]^-{f}&&M}
\]
commute. The category of monoids in $\cC$ with their homomorphisms is denoted by $\Mon(\cC)$.

\begin{examples}\label{Ex:Monoids}
\item For $\Set$ with its cartesian structure, $\Mon(\Set)=\Mon$, the usual category of monoids with their homomorphisms.
\item A unital ring $R$ is an abelian group that is also a monoid in which the distributive laws hold, that is, the multiplication $R\times R\to R$ is $\mathbb{Z}$-bilinear and is therefore equivalently described as a group homomorphism $R\otimes_{\mathbb{Z}}R\to R$. Hence, unital rings are precisely the monoids in $\AbGrp$ (with its usual tensor product), and $\Mon(\AbGrp)=\Rng$ is the category of unital rings and their homomorphisms.
\item A quantale $V$ is a complete lattice with a monoid operation $(-\otimes-):V\times V\to V$ that preserves suprema in each variable; with the tensor product in $\Sup$, the category of complete lattices and sup-preserving maps, the monoid operation may equivalently be considered a morphism $V\otimes V\to V$ in $\Sup$. Hence, one has $\Mon(\Sup)=\Qnt$, the category of quantales and their homomorphisms.
\end{examples}

\subsection{Actions in a monoidal category}
Let $\cC$ be a monoidal category, and $M=(M,m,e)$ a monoid in $\cC$. An \df{$M$-action} (more precisely, a \df{left $M$-action}) is an object $A$ in $\cC$ that comes with a $\cC$-morphism
\[
a:M\otimes A\to A
\]
that makes the following diagrams commute:
\[
\xymatrix@C=6ex{M\otimes(M\otimes A)\ar[d]_-{(m\otimes 1)\cdot\alpha^\inv}\ar[r]^-{1\otimes a}&M\otimes A\ar[d]^{a}\\
M\otimes A\ar[r]^-{a}&A}\qquad\qquad
\xymatrix@C=6ex{E\otimes A\ar[r]^-{e\otimes 1}\ar[dr]_-{\lambda}&M\otimes A\ar[d]^{a}\\
&A\rlap{\ .}}
\]
A $\cC$-morphism $f:A\to B$ between $M$-actions $(A,a)$ and $(B,b)$ is \df{equivariant} if the diagram
\[
\xymatrix@C=6ex{M\otimes A\ar[d]_{a}\ar[r]^{1\otimes f}&M\otimes B\ar[d]^{b}\\
A\ar[r]^f&B}
\]
commutes. The category of $M$-actions and equivariant $\cC$-morphisms is denoted by $\cC^M$, a notation that is motivated by the following result.

\begin{prop}\label{prop:ActionEM}
A monoid $(M,m,e)$ in a monoidal category $\cC$ gives rise to a monad $M=(M\otimes(-),\tilde{m},\tilde{e})$ on $\cC$, where 
\[
\tilde{m}_A=(m\otimes 1_A)\cdot\alpha_{M,M,A}^\inv\qquad\text{and}\qquad\tilde{e}_A=(e\otimes 1_A)\cdot\lambda_A^\inv
\]
for all $A\in\ob\cC$. The Eilenberg--Moore category $\cC^M$ of this monad is the category of $M$-actions and equivariant $\cC$-morphisms.
\end{prop}

\begin{proof}
Direct verifications.
\end{proof}

\begin{examples}\label{ex:Actions}
\item If $\cC=\Set$, and $M$ is a monoid, the category $\Set^M$ is the usual category of $M$-actions and equivariant maps.
\item The monoidal structure of $\AbGrp$ is given by the tensor product over $\mathbb{Z}$, and a monoid $R$ in $\AbGrp$ is a ring. Hence, $\AbGrp^R$ is the usual category $R\text{-}\Mod$ of left $R$-modules.
\item\label{ex:Actions:P} Given a quantale $\V=(V,\otimes,k)$, that is, a monoid is $\Sup$, the category $\Sup^\V$ is described as follows. A $\V$-action $X$ in $\Sup$ is a complete lattice $X$ together with a bimorphism $(-)\cdot(-):\V\times X\to X$ in $\Sup$ such that
\[
(u\otimes v)\cdot x=u\cdot(v\cdot x)\quad,\qquad k\cdot x=x\ ,
\]
for all $v\in\V$, $x\in X$, and a sup-map $f:X\to Y$ is equivariant whenever
\[
f(v\cdot x)=v\cdot f(x)
\]
for all $v\in\V$, $x\in X$.
\end{examples}

%

\subsection{Monadic actions}
In general, monadic functors do not compose. In the case of actions in $\cC^\mT$ however, they do (Theorem~\ref{thm:ActionsAreMonadic} below).

\begin{prop}\label{prop:MonadAction}
Let $(\mT,\kappa)$ be a monoidal monad on a monoidal category $(\cC,\otimes,E)$ such that $\cC^\mT$ has tensors, and $(\cC^\mT,\boxtimes,TE)$ a monoidal category whose structure morphisms are induced by those of $(\cC,\otimes,E)$.

A monoid $(M,\xi)$ in $(\cC^\mT,\boxtimes,TE)$ induces a monad $M\boxtimes\mT$ on $\cC$ whose functor is $M\boxtimes T(-):\cC\to\cC$, and whose multiplication and unit are given by their components at $X\in\ob\cC$ as follows:
\begin{align*}
\tilde{\mu}_X&=(m\boxtimes 1_{TX})\cdot\overline{\alpha}_{M,M,TX}^\inv\cdot(1_M\boxtimes(\xi\bowtie\mu_X))\ ,\\
\tilde{\eta}_X&=(e\boxtimes 1_{TX})\cdot\overline{\lambda}_{TX}^\inv\cdot\eta_X\ .
\end{align*}
\end{prop}

\begin{proof}
The composite of the adjunctions
\[
\xymatrix@C=3.5em{(\cC^\mT)^M\ar[r]<-1ex>\ar@{}[r]|-{\bot}&\cC^\mT\ar[r]<-1ex>\ar@{}[r]|-{\bot}\ar[l]<-1ex>_-{M\boxtimes(-)}&\cC\ar[l]<-1ex>_-{T}}
\]
yields an adjunction $\xymatrix@C=3.5em{(\cC^\mT)^M\ar[r]<-1ex>\ar@{}[r]|-{\bot}&\cC\ar[l]<-1ex>_-{M\boxtimes T(-)}}$ that induces the described monad.
\end{proof}

\begin{lem}\label{lem:MndMorphToAction}
Let $(\mT,\kappa)$ be a monoidal monad on a monoidal category $(\cC,\otimes,E)$ such that $\cC^\mT$ has tensors, and $(\cC^\mT,\boxtimes,TE)$ a monoidal category whose structure morphisms are induced by those of $(\cC,\otimes,E)$.

The $\cC$-morphisms $(e\boxtimes 1_{TX})\cdot\overline{\lambda}_{TX}^\inv:TX\to M\boxtimes TX$ are the components of a monad morphism $\tau:\mT\to M\boxtimes\mT$.
\end{lem}

\begin{proof}
Naturality of the components $\tau_{X}:=(e\boxtimes 1_{TX})\cdot\overline{\lambda}_{TX}^\inv$ (for $X\in\ob\cC$) is immediate. Agreement on the units is also easily verified:
\[
\tilde{\eta}_X=(e\boxtimes 1_{TX})\cdot\overline{\lambda}_{TX}^\inv\cdot\eta_X=\tau_X\cdot\eta_X\ .
\]
For the multiplications, one has
\begin{align*}
\tilde{\mu}_X\cdot(1_M\boxtimes T\tau_X)\cdot\tau_{TX}&=(m\boxtimes 1_{TX})\cdot\overline{\alpha}_{M,M,TX}^\inv\cdot(1_M\boxtimes((e\boxtimes 1_{TX})\cdot\overline{\lambda}_{TX}^\inv\cdot\mu_X))\cdot\tau_{TX}\\
&=(\overline{\lambda}_M\boxtimes 1_{TX})\cdot((1_{TE}\boxtimes e)\boxtimes 1_{TX})\cdot\overline{\alpha}_{TE,TE,TX}^\inv\cdot(1_{TE}\boxtimes(\overline{\lambda}_{TX}^\inv\cdot\mu_X))\cdot\overline{\lambda}_{TX}^\inv\\
&=(e\boxtimes 1_{TX})\cdot(\overline{\lambda}_{TE}\boxtimes 1_{TX})\cdot\overline{\alpha}_{TE,TE,TX}^\inv\cdot\overline{\lambda}_{TE\boxtimes TX}^\inv\cdot\overline{\lambda}_{TX}^\inv\cdot\mu_X\\
&=(e\boxtimes 1_{TX})\cdot\overline{\lambda}_{TX}^\inv\cdot\mu_X=\tau_X\cdot\mu_X
\end{align*}
(we use \cite{Kel:64} for the penultimate equality), so $\tau$ is indeed a monad morphism.
\end{proof}

\begin{thm}\label{thm:ActionsAreMonadic}
Let $(\mT,\kappa)$ be a monoidal monad on a monoidal category $(\cC,\otimes,E)$ such that $\cC^\mT$ has tensors, and $(\cC^\mT,\boxtimes,TE)$ a monoidal category whose structure morphisms are induced by those of $(\cC,\otimes,E)$.

For a monoid $(M,\xi)$ in $(\cC^\mT,\boxtimes,TE)$, there is an isomorphism between the category of algebras of the monad $M\boxtimes\mT$ and the category of $M$-actions in $\cC^\mT$:
\[
\cC^{M\boxtimes\mT}\cong(\cC^\mT)^M\ .
\]
In particular, the forgetful functor $(\cC^\mT)^M\to\cC$ is strictly monadic.
\end{thm}

\begin{proof}
The comparison functor $K:(\cC^\mT)^M\to\cC^{M\boxtimes\mT}$ sends  a $\mT$-algebra $(A,a_1)$ with an action $a_2:M\boxtimes A\to A$ in $\cC^\mT$ to the $(M\boxtimes\mT)$-algebra $(A,a_2\cdot(1_M\boxtimes a_1))$. We proceed to verify that $K$ is an isomorphism.

The monad morphism $\tau:\mT\to M\boxtimes\mT$ of Lemma~\ref{lem:MndMorphToAction} induces a functor $\cC^\tau:\cC^{M\boxtimes\mT}\to\cC^\mT$ that sends a $(M\boxtimes\mT)$-algebra $(A,a)$ to the $\mT$-algebra $(A,a\cdot\tau_A)$ and commutes with the forgetful functors to $\cC$. Set $a_1:=a\cdot\tau_A$, so that 
\[
\xymatrix@C=3.5em{T(TM\otimes TA)\ar[r]<0.7ex>^-{\mu\cdot T\kappa}\ar[r]<-0.7ex>_-{T(\xi\otimes a_1)}&T(M\otimes A)\ar[r]^-{q}&M\boxtimes A}
\]
is a coequalizer diagram. There is then a unique $\cC^\mT$-morphism $a_2:M\boxtimes A\to A$ such that
\[
\xymatrix{T(M\otimes A)\ar[rr]^-{q}\ar[d]_{T(1\otimes\eta)}&&M\boxtimes A\ar[d]^{a_2}\\
T(M\otimes TA)\ar[r]^-{q}&M\boxtimes TA\ar[r]^-{a}&A}
\]
commutes. To see this, we use the universal property of $q_{M,A}$. Indeed, $\tau_{A}=(e\boxtimes 1_{TA})\cdot\overline{\lambda}_{TA}^\inv:TA\to M\boxtimes TA$ is a $\cC^\mT$-morphism, so the diagram
\[
\xymatrix@C=7.5em{T(TM\otimes TA)\ar[r]^-{T(\xi\otimes\eta)}\ar[dr]_{T(\xi\otimes 1)}&T(M\otimes TTA)\ar[r]^-{T(1\otimes T((e\boxtimes 1)\cdot\overline{\lambda}^\inv))}\ar[d]^{T(1\otimes \mu)}&T(M\otimes T(M\boxtimes TA))\ar[d]^{T(1\otimes(\xi\bowtie\mu))}\\
&T(M\otimes TA)\ar[r]^-{T(1\otimes((e\boxtimes 1)\cdot\overline{\lambda}^\inv))}&T(M\otimes(M\boxtimes TA))}
\]
commutes; by definition of a $(M\boxtimes\mT)$-algebra $(A,a)$, the diagram
\[
\xymatrix@C=4em{T(M\otimes T(M\boxtimes TA))\ar[d]_-{T(1\otimes(\xi\bowtie\mu))}\ar[r]^-{q}&M\boxtimes T(M\boxtimes TA)\ar[d]_{1\boxtimes(\xi\bowtie\mu)}\ar[r]^-{1\boxtimes Ta}&M\boxtimes TA\ar[dr]+<-1ex,1.5ex>^-{a}&\\
T(M\otimes(M\boxtimes TA))\ar[r]^-{q}&M\boxtimes(M\boxtimes TA))\ar[r]^-{(m\boxtimes 1)\cdot\overline{\alpha}^\inv}&M\boxtimes TA\ar[r]^-{a}&A}
\]
also commutes. By glueing these diagrams together along $T(1_M\otimes(\xi\bowtie\mu_A))$, one obtains, with $a_1=a\cdot(e\boxtimes 1_{TA})\cdot\overline{\lambda}_{TA}^\inv$ and $q_{M,TA}\cdot T(1\otimes Ta)=(1_M\boxtimes Ta)\cdot q_{M,T(M\boxtimes TA)}$,
\begin{align*}
a\cdot q_{M,TA}&\cdot T(1_M\otimes\eta_A)\cdot T(\xi\otimes a_1)\\
&=a\cdot q_{M,TA}\cdot T(1_M\otimes Ta_1)\cdot T(\xi\otimes\eta_{TA})\\
&=a\cdot (1_M\boxtimes Ta)\cdot q_{M,T(M\boxtimes TA)}\cdot T(1_M\otimes T((e\boxtimes 1_{TA})\cdot\overline{\lambda}_{TA}^\inv))\cdot T(\xi\otimes\eta_{TA})\\
&=a\cdot(m\boxtimes 1_{TA})\cdot\overline{\alpha}_{M,M,TA}^\inv\cdot q_{M,M\boxtimes TA}\cdot T(1_M\otimes((e\boxtimes 1_{TA})\cdot\overline{\lambda}_{TA}^\inv))\cdot T(\xi\otimes 1_{TA})\\
&=a\cdot(m\boxtimes 1_{TA})\cdot\overline{\alpha}_{M,M,TA}^\inv\cdot(1_M\boxtimes((e\boxtimes 1_{TA})\cdot\overline{\lambda}_{TA}^\inv))\cdot q_{M,TA}\cdot T(\xi\otimes 1_{TA})\\
&=a\cdot(\overline{\rho}_{M}\boxtimes 1_{TA})\cdot\overline{\alpha}_{M,TE,TA}^\inv\cdot(1_M\boxtimes\overline{\lambda}_{TA}^\inv)\cdot q_{M,TA}\cdot T(\xi\otimes 1_{TA})\\
&=a\cdot q_{M,TA}\cdot T(\xi\otimes 1_{TA})\ .
\intertext{Since one also has}
a\cdot q_{M,TA}&\cdot T(1_M\otimes\eta_A)\cdot \mu_{M\otimes A}\cdot T\kappa_{M,A}\\
&=a\cdot q_{M,TA}\cdot\mu_{M\otimes TA}\cdot T\kappa_{M,TA}\cdot T(1_{TM}\otimes T\eta_A)\\
&=a\cdot q_{M,TA}\cdot T(\xi\otimes 1_{TA})\ , 
\end{align*}
the existence and unicity of the required $\cC^\mT$-morphism $a_2$ follows. Moreover, 
\begin{align*}
a\cdot q_{M,TA}&=a\cdot q_{M,TA}\cdot T(1_M\otimes\eta_A)\cdot T(1_M\otimes a_1)\\
&=a_2\cdot q_{M,A}\cdot T(1_M\otimes a_1)\\
&=a_2\cdot(1_M\boxtimes a_1)\cdot q_{M,TA}\ ,
\end{align*}
so $a=a_2\cdot(1_M\boxtimes a_1)$ because $q_{M,TA}$ is epic. Let us verify that $a_2:M\boxtimes A\to A$ defines an action. For this, we use
\begin{align*}
a\cdot(1_M\boxtimes Ta)&=a_2\cdot(1_M\boxtimes a_1)\cdot (1_M\boxtimes Ta)\\
&=a_2\cdot(1_M\boxtimes a)\cdot (1_M\boxtimes (\xi\bowtie\mu_A))\\
&=a_2\cdot(1_M\boxtimes a_2)\cdot (1_M\boxtimes(1_M\boxtimes a_1))\cdot (1_M\boxtimes (\xi\bowtie\mu_A))
\intertext{and}
a\cdot\tilde{\mu}_A&=a_2\cdot(m\boxtimes 1_{A})\cdot (1_{M\boxtimes M}\boxtimes a_1)\cdot\overline{\alpha}_{M,M,TA}^\inv\cdot(1_M\boxtimes(\xi\bowtie\mu_A))\\
&=a_2\cdot(m\boxtimes 1_{A})\cdot\overline{\alpha}_{M,M,A}^\inv\cdot(1_M\boxtimes(1_M\boxtimes a_1))\cdot(1_M\boxtimes(\xi\bowtie\mu_A))\ .
\end{align*}
Moreover, for any $\mT$-algebra structure $b:TB\to B$, the $\cC^\mT$-morphisms $T(1_M\otimes b)$ is an epimorphism, so that $q_{M,A}\cdot T(1_M\otimes b)=(1_M\boxtimes b)\cdot q_{M,TA}$ implies that $(1_M\boxtimes b)$ is epic. Since $a\cdot(1_M\boxtimes Ta)=a\cdot\tilde{\mu}_A$, and $(1_M\boxtimes(1_M\boxtimes a_1))$, $(1_M\boxtimes(\xi\bowtie\mu_A))$ are both epic, we obtain
\[
a_2\cdot(1_M\boxtimes a_2)=a_2\cdot(m\boxtimes 1_A)\cdot\overline{\alpha}_{M,M,A}^\inv\ ,
\]
the first condition for $a_2$ to be an action. The second condition comes from
\[
a\cdot\tilde{\eta}_A=a_2\cdot(1_M\boxtimes a_1)\cdot(e\boxtimes 1_{TA})\cdot\overline{\lambda}_{TA}^\inv\cdot\eta_A=a_2\cdot(e\boxtimes 1_{A})\cdot\overline{\lambda}_{A}^\inv\cdot a_1\cdot\eta_A=a_2\cdot(e\boxtimes 1_{A})\cdot\overline{\lambda}_{A}^\inv
\]
with $a\cdot\tilde{\eta}_A=1_A$. Finally, any $(M\boxtimes\mT)$-algebra homomorphism $f:(A,a)\to(B,b)$ yields a $\mT$-algebra homomorphism $f:(A,a\cdot\tau_A)\to(B,b\cdot\tau_B)$ that is equivariant: one has
\[
f\cdot a_2\cdot(1_M\boxtimes a_1)=f\cdot a=b\cdot (1_M\boxtimes Tf)=b_2\cdot (1_M\boxtimes b_1)\cdot (1_M\boxtimes Tf)=b_2\cdot (1_M\boxtimes f)\cdot (1_M\boxtimes a_1)\ ,
\]
so that $f\cdot a_2=b_2\cdot (1_M\boxtimes f)$ because $(1_M\boxtimes a_1)$ is epic.

Hence, a $(M\boxtimes\mT)$-algebra $(A,a)$ yields a $\mT$-algebra $(A,a_1)$ with an action $a_2:M\boxtimes A\to A$, and $K((A,a_1),a_2)$ returns the original $(M\boxtimes\mT)$-algebra $(A,a)$, since $a=a_2\cdot(1_M\boxtimes a_1)$. 

Conversely, the $K$-image of a $\mT$-algebra $(A,a_1)$ with an action $a_2:M\boxtimes A\to A$ is a $(M\boxtimes\mT)$-algebra $(A,a_2\cdot(1_M\boxtimes a_1))$. Since
\[
a_2\cdot(1_M\boxtimes a_1)\cdot\tau_A=a_2\cdot (e\boxtimes 1_A)\cdot\overline{\lambda}_A^\inv\cdot a_1=a_1
\]
and the diagram 
\[
\xymatrix{T(M\otimes A)\ar[r]^-{1}\ar[d]_{T(1\otimes\eta)}&T(M\otimes A)\ar[rr]^-{q}&&M\boxtimes A\ar[d]^{a_2}\\
T(M\otimes TA)\ar[rr]^(0.55){q}\ar[ur]_(0.6){\ \,T(1\otimes a_1)}&&M\boxtimes TA\ar[ur]^(0.4){1\boxtimes a_1}\ar[r]^-{a}&A}
\]
commutes, one recuperates the original triplet $((A,a_1),a_2)$ from $(A,a_2\cdot(1_M\boxtimes a_1))$; that is, $K$ is an isomorphism.
\end{proof}

\begin{examples}\label{ex:ApplicThmAction}
\item\label{ex:ApplicThmAction:I} Any monoid $M$ in a monoidal category $(\cC,\otimes,E)$ yields an isomorphism $(\cC^\mId)^M\cong\cC^M\cong\cC^{M\otimes\mId}$. This immediate result is also the $\mT=\mId$ case of Theorem~\ref{thm:ActionsAreMonadic}.
\item\label{ex:ApplicThmAction:Ab} If $\mT=\mAb$ is the free abelian group monad and $M=R$ is a ring, the isomorphisms
\[
R\text{-}\Mod\cong\AbGrp^R\cong(\Set^\mAb)^R\cong\Set^{R\otimes_{\mathbb{Z}}\mAb}
\]
recall the classical monadicity of $R$-modules over $\Set$, and describe the free $R$-module over a set $X$ as $R\otimes_{\mathbb{Z}}\mathit{Ab}X$.
\item\label{ex:ApplicThmAction:P} For a quantale $V$, the category $\Sup^\V$ of $\V$-actions (Example~\ref{ex:Actions}\ref{ex:Actions:P}) is isomorphic to the category $\Set^{V\otimes\mP}$. The classical description of the tensor in $\Sup$ (see for example \cite{BanNel:76}) yields isomorphisms
\[
V\otimes PX\cong\Sup(\Sup(V,\Sup(PX,2)),2)\cong\Set(X,V)\ .
\]
The case where $V$ is integral was treated in \cite{PedTho:89}, where it is proved that $\Sup^\V$ is monadic over $\Set$.
\end{examples}


\subsection{Monad morphisms}
Suppose that $\mS$ and $\mT$ are monoidal monads on $(\cC,\otimes, E)$ that respectively induce monoidal categories $(\cC^\mS,\circledast,SE)$ and $(\cC^\mT,\boxtimes,TE)$. Proposition~\ref{prop:MonoidAndMonadMorphInducing} below shows that every pair of monoid homomorphism $f:N\to M$ and monoidal monad morphism $\phi:\mS\to\mT$ induces a monad morphism $(f,\phi):N\circledast\mS\to M\boxtimes\mT$, and thus a functor $(\cC^\mT)^M\to(\cC^\mS)^N$ between the respective categories of actions. The usual  \emph{restriction-of-scalars} functor between categories of modules then appears as the $\mS=\mT=\mAb$ instance of this result (Corollary~\ref{cor:RestScal}). 

\begin{lem}\label{lem:MonoidMonadTransportsMonoids}
Let $(\cC,\otimes,E)$ be a monoidal category, and $(\mS,\iota)$, $(\mT,\kappa)$ monoidal monads on $\cC$. Suppose that $\cC^\mT$ and $\cC^\mS$ both have tensors that make $(\cC^\mT,\boxtimes,TE)$ and $(\cC^\mS,\circledast,SE)$ monoidal categories with structures induced by those of $(\cC,\otimes,E)$.

If $\phi:\mS\to\mT$ is a monoidal monad morphism and $((N,\zeta),n,d)$ is a monoid in $\cC^\mT$, then $(N,\zeta\cdot\phi_N)$ with multiplication $n\cdot\overline{\phi}_{N,N}$ and unit $d\cdot\phi_{E}$ is a monoid in $\cC^\mS$.
\end{lem}

\begin{proof}
Since $(N,\zeta)$ is a $\mT$-algebra, $(N,\zeta\cdot\phi_N)$ is an $\mS$-algebra. The structures $d\cdot\phi_{E}:SE\to N$ and $n\cdot\overline{\phi}_{N,N}:N\circledast N\to N$ are $\cC^\mS$-morphisms (see  the proof of Proposition~\ref{prop:MonadMorphMonoidal}). Commutativity of the corresponding monoid diagrams in $(\cC^\mS,\circledast,SE)$ follows from commutativity of the monoid diagrams of $(N,n,d)$ in $(\cC^\mT,\boxtimes,TE)$ combined with commutativity of the diagrams showing that the functor $\cC^\phi:\cC^\mT\to\cC^\mS$ is monoidal with respect to $\overline{\phi}$ and $\phi_E$ (Proposition~\ref{prop:MonadMorphMonoidal}).
\end{proof}

\begin{prop}\label{prop:MonoidAndMonadMorphInducing}
Let $(\cC,\otimes,E)$ be a monoidal category, and $(\mS,\iota)$, $(\mT,\kappa)$ monoidal monads on $\cC$. Suppose that $\cC^\mT$ and $\cC^\mS$ both have tensors that make $(\cC^\mT,\boxtimes,TE)$ and $(\cC^\mS,\circledast,SE)$ monoidal categories with structures induced by those of $(\cC,\otimes,E)$.

If $\phi:\mS\to\mT$ is a monoidal monad morphism, and $f:N\to M$ a monoid homomorphism in $\cC^\mT$, then there is a monad morphism $(f,\phi):N\circledast\mS\to M\boxtimes\mT$ whose components at $X\in\ob\cC$ are given by the $\cC^\mS$-morphism $\overline{\phi}_{M,TX}\cdot(f\circledast\phi_{X}):N\circledast SX\to M\boxtimes TX$.
\end{prop}

\begin{proof}
By Lemma~\ref{lem:MonoidMonadTransportsMonoids}, $N$ can be seen as a monoid in $(\cC^\mS,\circledast,SE)$, thus defining a monad $N\circledast\mS$. Via the functor $\cC^\phi:\cC^\mT\to\cC^\mS$, the arrows $f$ and $\phi_X$ are $\cC^\mS$-morphisms, so $f\circledast\phi_{X}$ is defined in $\cC^\mS$, and so is $\overline{\phi}_{M,TX}$ by Proposition~\ref{prop:MonadMorphMonoidal}. To verify that these components define a monad morphism, we use our usual notations $\mS=(S,\nu,\delta)$, $\mT=(T,\mu,\eta)$ for the monads, and $(N,n,d)$, $(M,m,e)$ for the monoids. Moreover, we let the context differentiate between the induced structure morphisms $\overline{\alpha}$, $\overline{\delta}$ and $\overline{\rho}$ of $\cC^\mT$ or $\cC^\mS$.

By using the definitions and properties of the involved morphisms, we compute
\begin{align*}
\overline{\phi}_{M,TX}\cdot(f\circledast\phi_{X})\cdot\tilde{\delta}_X
&=\overline{\phi}_{M,TX}\cdot(f\circledast\phi_{X})\cdot(d\cdot\phi_{E}\circledast 1_{SX})\cdot\overline{\lambda}_{SX}^\inv\cdot\delta_X\\
&=\overline{\phi}_{M,TX}\cdot(e\circledast 1_{TX})\cdot(\phi_{E}\circledast 1_{TX})\cdot(1_{SE}\circledast\phi_X)\cdot\overline{\lambda}_{SX}^\inv\cdot\delta_X\\
&=(e\boxtimes 1_{TX})\cdot\overline{\phi}_{TE,TX}\cdot(\phi_{E}\circledast 1_{TX})\cdot\overline{\lambda}_{TX}^\inv\cdot\phi_X\cdot\delta_X\\
&=(e\boxtimes 1_{TX})\cdot\overline{\lambda}_{TX}^\inv\cdot\eta_X=\tilde{\eta}_X\ .
\end{align*}
For the multiplications, we use that
\begin{align*}
\overline{\phi}_{M,TX}&\cdot(f\circledast\phi_{X})\cdot(n\cdot\overline{\phi}_{N,N}\circledast 1_{SX})\cdot\overline{\alpha}_{N,N,SX}^\inv\\
&=(m\cdot(f\boxtimes f)\boxtimes 1_{TX})\cdot\overline{\phi}_{N\boxtimes N,TX}\cdot(\overline{\phi}_{N,N}\circledast 1_{TX})\cdot\overline{\alpha}_{N,N,TX}^\inv\cdot(1_N\circledast(1_N\circledast\phi_X))\\
&=(m\cdot(f\boxtimes f)\boxtimes 1_{TX})\cdot\overline{\alpha}_{N,N,TX}^\inv\cdot\overline{\phi}_{N,N\boxtimes TX}\cdot(1_N\circledast\overline{\phi}_{N,TX})\cdot(1_N\circledast(1_N\circledast\phi_X))\\
&=(m\boxtimes 1_{TX})\cdot\overline{\alpha}_{M,M,TX}^\inv\cdot\overline{\phi}_{M,M\boxtimes TX}\cdot(1_M\circledast\overline{\phi}_{M,TX})\cdot(f\circledast(f\circledast\phi_X))
\intertext{and since $\overline{\phi}_{M,TX}\cdot(f\circledast\phi_X):N\circledast SX\to M\boxtimes TX$ is a $\cC^\mS$-morphism,}
\overline{\phi}_{M,M\boxtimes TX}&\cdot(1_M\circledast\overline{\phi}_{M,TX})\cdot(f\circledast(f\circledast\phi_X))\cdot(1_N\circledast(\zeta\cdot\phi_N\bowtie\nu_X))\\
&=\overline{\phi}_{M,M\boxtimes TX}\cdot(f\circledast((\xi\bowtie\mu_X)\cdot\phi_{M\circledast TX}))\cdot(1_N\circledast S(\overline{\phi}_{M,TX}\cdot(f\circledast\phi_X)))\\
&=(1_M\boxtimes(\xi\bowtie\mu_X))\cdot\overline{\phi}_{M,T(M\boxtimes TX)}\cdot(f\circledast\phi_{M\boxtimes TX})\cdot(1_N\circledast S(\overline{\phi}_{M,TX}\cdot(f\circledast\phi_X)))\ .
\end{align*}
Hence,
\begin{align*}
\overline{\phi}_{M,TX}\cdot(f\circledast\phi_{X})\cdot\tilde{\nu}_X&=\overline{\phi}_{M,TX}\cdot(f\circledast\phi_{X})\cdot(n\cdot\overline{\phi}_{N,N}\circledast 1_{SX})\cdot\overline{\alpha}_{N,N,SX}^\inv\cdot(1_N\circledast(\zeta\cdot\phi_N\bowtie\nu_X))\\
&=\tilde{\mu}_X\cdot\overline{\phi}_{M,T(M\boxtimes TX)}\cdot(f\circledast\phi_{M\boxtimes TX})\cdot(1_N\circledast S(\overline{\phi}_{M,TX}\cdot(f\circledast\phi_X)))\ ,
\end{align*}
so the components $\overline{\phi}_{M,TX}\cdot(f\circledast\phi_X)$ satisfy the two conditions for being a monad morphism.
\end{proof}

\begin{cor}[Restriction of scalars]\label{cor:RestScal}
Let $(\cC,\otimes,E)$ be a monoidal category, and $(\mT,\kappa)$ a monoidal monad on $\cC$. Suppose that $\cC^\mT$ has tensors that make $(\cC^\mT,\boxtimes,TE)$ a monoidal category with structure morphisms induced by those of $(\cC,\otimes,E)$. 

If $f:N\to M$ is a monoid homomorphism in $\cC^\mT$, then the monad morphism $(f,1_\mT):N\boxtimes\mT\to M\boxtimes\mT$ induces a functor $\cC^{(f,1_\mT)}:(\cC^\mT)^M\to(\cC^\mT)^N$.
\end{cor}

\begin{proof}
The monad morphism is the $\phi=1_\mT$ case of Proposition~\ref{prop:MonoidAndMonadMorphInducing}, and the result follows from the isomorphisms $(\cC^\mS)^N\cong\cC^{N\circledast\mS}$ and $(\cC^\mT)^M\cong\cC^{M\boxtimes\mT}$ (Theorem~\ref{thm:ActionsAreMonadic}).
\end{proof}

\subsection*{Acknowledgement}
I would like to heartily thank Fred Linton that inspired this work by his insightful Remark in \cite{Lin:69}, and followed it with a kind and enlightening e-mail correspondence more than 40 years later \cite{Lin:11}. Dirk Hofmann also helped spot an error in an earlier version of this work, and Martin Brandenburg provided useful comments and additional references: thanks to them, too.


\bibliographystyle{plain}

\end{document}